\documentclass{article}
\usepackage[utf8]{inputenc}
\usepackage{amsmath}
\usepackage{amsfonts}
\usepackage{amssymb}
\usepackage{amsthm}
\usepackage{mathtools}
\usepackage[capitalize]{cleveref}
\usepackage[style=alphabetic]{biblatex}
\usepackage{tikz-cd}

\theoremstyle{definition}
\newtheorem{definition}{Definition}[section]

\theoremstyle{plain}
\newtheorem{prop}[definition]{Proposition}
\newtheorem{lemma}[definition]{Lemma}
\newtheorem{thm}[definition]{Theorem}
\newtheorem{cor}[definition]{Corollary}

\theoremstyle{remark}
\newtheorem{remark}[definition]{Remark}

\numberwithin{equation}{definition}

\renewcommand{\emptyset}{\varnothing}

\DeclareMathOperator{\Hom}{Hom}
\DeclareMathOperator{\End}{End}
\DeclareMathOperator{\affo}{aff}
\DeclareMathOperator{\fino}{fin}
\DeclareMathOperator{\domin}{dom}

\newcommand{\aff}[1]{#1^{\affo}}
\newcommand{\fin}[1]{#1^{\fino}}

\newcommand{\q}[1]{#1_q}
\newcommand{\mq}[1]{#1_{q^{-\frac{1}{2}}}}
\newcommand{\pmq}[1]{#1_{q^{\pm1}}}
\newcommand{\rtq}[1]{#1_{q^{\pm\frac{1}{2}}}}
\newcommand{\any}[1]{#1_{\bullet}}
\newcommand{\astot}[1]{#1_{\infty}}

\newcommand{\z}{\mathbb{Z}}
\newcommand{\zq}{\q{\z}}
\newcommand{\zmq}{\mq{\z}}
\newcommand{\zpmq}{\pmq{\z}}
\newcommand{\zrtq}{\rtq{\z}}
\newcommand{\zany}{\any{\z}}

\newcommand{\weyl}{W}
\newcommand{\affw}{\aff{\weyl}}

\newcommand{\gen}{S}
\newcommand{\affgen}{\aff{\gen}}
\newcommand{\fingen}{\fin{\gen}}

\newcommand{\hec}{\mathcal{H}}

\newcommand{\hecmq}{\mq{\hec}}
\newcommand{\hecpmq}{\pmq{\hec}}
\newcommand{\hecrtq}{\rtq{\hec}}
\newcommand{\hecany}{\any{\hec}}
\newcommand{\hecinf}{\astot{\hec}}

\newcommand{\hecirt}{\rtq{{\astot{\hec}}}}
\newcommand{\hecia}{\any{{\astot{\hec}}}}

\newcommand{\affhecrtq}{\aff{\hecrtq}}
\newcommand{\affhecany}{\aff{\hecany}}

\newcommand{\sch}{\mathcal{S}}

\newcommand{\schmq}{\mq{\sch}}
\newcommand{\schpmq}{\pmq{\sch}}
\newcommand{\schrtq}{\rtq{\sch}}
\newcommand{\schany}{\any{\sch}}
\newcommand{\schinf}{\astot{\sch}}
\newcommand{\schipm}{\pmq{{\astot{\sch}}}}
\newcommand{\schirt}{\rtq{{\astot{\sch}}}}
\newcommand{\schia}{\any{{\astot{\sch}}}}

\newcommand{\topcell}{\mathcal{B}}
\newcommand{\tcany}{\any{\topcell}^{\lambda}}
\newcommand{\tcrtq}{\rtq{\topcell}^{\lambda}}
\newcommand{\tcu}{\tilde{\topcell}}
\newcommand{\tcuany}{\any{\tcu}^{\lambda}}
\newcommand{\tcuq}{\q{\tcu}^{\lambda}}
\newcommand{\tcupmq}{\pmq{\tcu}^{\lambda}}
\newcommand{\tcurtq}{\rtq{\tcu}^{\lambda}}

\newcommand{\kln}{C}
\newcommand{\klu}{\tilde{\kln}}
\newcommand{\kla}{\hat{\kln}}
\newcommand{\scn}{h}
\newcommand{\scu}{\tilde{\scn}}
\newcommand{\asn}{t}
\newcommand{\asu}{\tilde{\asn}}
\newcommand{\ascn}{\gamma}
\newcommand{\ascu}{\tilde{\ascn}}

\title{Affine Hecke and Schur algebras of type A without a square root of q}
\author{Rose Berry}
\date{Max Plank Institute for Mathematics, Bonn}

\begin{document}

\maketitle

\begin{abstract}

We provide an affine cellular structure on the extended affine Hecke algebra and affine $q$-Schur algebra of type $A_{n-1}$ that is defined over $\z\left[q^{\pm1}\right]$, that is, without an adjoined $q^{\frac{1}{2}}$. This is with an eye to applications in the representation theory of $\mathrm{GL}_n(F)$ for a $p$-adic field $F$ over coefficient rings in which $p$ is invertible but does not have a square root, which have been a topic of recent interest. This is achieved via a renormalisation of the known affine cellular structure over $\z\left[q^{\pm\frac{1}{2}}\right]$ at each left and right cell, which is chosen to ensure that the diagonal intersections remain subalgebras and that the left and right cells remain isomorphic. We furthermore show that the affine cellular structure on the Schur algebra has idempotence properties which imply finite global dimension, an important ingredient for the applications to representations of $p$-adic groups.
    
\end{abstract}

\section{Introduction}

The extended affine Hecke algebra of type $\tilde{A}_{n-1}$ appears throughout representation theory. Most notably, its complex modules are equivalent to the principal block of the category of smooth complex representations of $\mathrm{GL}_n(F)$ for a $p$-adic field $F$ (\cite{bushnell1999semisimple}), and its anti-spherical module is isomorphic to the Grothendieck group of the principal block of rational representations of $\mathrm{GL}_n(k)$ for $k$ algebraically closed of characteristic $p$, as well as to the Grothendieck group of the principal block of representations of the affine quantum group $U(\hat{\mathfrak{gl}}_n)$ at a $p$-th root of unity (\cite{andersen1994representations}).

In extending these applications, the affine $q$-Schur algebra naturally arises. It is a quotient of both the group algebra of $\mathrm{GL}_n(F)$ (\cite{vigneras2003schur}), and a quotient of $U(\hat{\mathfrak{gl}}_n)$ (\cite{green1999affine}), in both cases via a double centraliser relation with the Hecke algebra. In the former, it describes a subcategory of the principal block of smooth representations over an algebraically closed field of characteristic $l\neq p$, and can be dg enhanced to describe the derived category of the whole block (\cite{berry2025derived}), while in the latter it can be used to give a cell structure on the modified quantum group (\cite{mcgerty2003cells}).

Both the Hecke and Schur algebras are defined over the integers adjoined a single abstract parameter $q$. However, traditionally both an inverse $q^{-1}$ and a square root $q^{\frac{1}{2}}$ are adjoined as well. This is because much of the theory, such as the construction of Lusztig's $a$-function (\cite{lusztig1985cells}) and the asymptotic algebra (\cite{lusztig1987cells}), require working over this extended ring. However, recent work in the study of $p$-adic groups (\cite{dat2009finitude,helm2024block,dat2025depth}) has considered coefficient rings in which $q$ does not have a square root. It thus makes sense to ask if the theory of Hecke and Schur algebras can be defined over this larger class of rings.

Specifically, we seek to define an affine cellular structure in the sense of \cite{koenig2012affine} on both the Hecke and Schur algebras over $\z\left[q^{\pm1}\right]$, and furthermore show that the latter satisfies additional idempotence properties that imply finite global dimension. This property is in ingredient in the author's proof in \cite{berry2025derived} that the derived principal block is equivalent to a dg enhanced Schur algebra, and so is a step in generalising this result to these more general coefficient rings. We follow the previous constructions of an affine cellular structure for both algebras (\cite{koenig2012affine, cui2015affine, cui2016affine}), dealing with obstructions as they arise. In the case of finite type $A$, the Schur algebra is already known to be quasihereditary, and hence cellular, without $q^{\frac{1}{2}}$ (\cite{du1998cells}). This is proven by normalising the usual isomorphisms between left cells in a two-sided cell to be defined integrally, and a similar argument forms a key step of our reasoning.

The extra difficulty posed on the affine setting is finding a suitable integral normalisation of the asymptotic algebra. The typical normalisation is defined using the normalised Kazhdan-Lusztig basis, which is only defined with an adjoined $q^{\frac{1}{2}}$. The module spanned by the unnormalised Kazhdan-Lusztig basis, which is defined integrally, is not a priori closed under multiplication. However, it turns out a closely related normalisation is closed, and thus provides a candidate for the asymptotic algebra in the affine cellular structure. Checking this amounts to understanding the lengths of the elements in each two-sided cell, and how they vary under multiplication.

Both above normalisation results depend on the existence of the adjoined inverse $q^{-1}$. It thus seems unlikely, at least with the existing approach, that this can also be removed. However, this is no great loss, as the relevant applications to $p$-adic groups only make sense in settings where $q^{-1}$ exists. Nonetheless, as cells themselves can be defined without $q^{-1}$, it would be interesting to see what structure they can be endowed with in full generality.

Care is taken throughout to distinguish between results proven for the affine Hecke algebra and those proven for the extended affine Hecke algebra, and to ensure that our constructions are correct in the generality we have stated them. We also try to present a complete overview of the construction, which was previously scattered throughout a considerable number of papers, which use differing conventions for the various objects within. Of particular note, we use the version of the Kazhdan-Lusztig basis that is a nonnegative linear combination of the standard basis, and we give the construction of the Schur algebra and its Kazhdan-Lusztig basis in terms of Coxeter group combinatorics instead of infinite matrices and perverse sheaves. 

Section two recalls the definitions and basic structure of the extended affine Weyl group, Hecke algebra, and Schur algebra. Section 3 recalls the definition of the normalised and unnormalised Kazhdan-Lusztig bases and of left, right, and two-sided cells, as well as existing results on their partial order and number. In Section 4 we recall the construction of the usual isomorphism between left cells in a two-sided cell, and show how it can be normalised in such a way as to be defined integrally. In Section 5 we recall the construction of the asymptotic algebra, as well as its known form, and then give an alternate normalisation which we prove is an integral subalgebra. In Section 6 we combine these two results with a final result, showing that left cells are a bimodule under the right action of our newly normalised asymptotic algebra, to give an affine cellular structure on the Hecke and Schur algebras. We furthermore show that the structure on the Schur algebra has idempotence properties that imply it has finite global dimension.

\section{The Weyl group, Hecke algebra, and Schur algebra}

In this section we recall the definitions of the extended affine Weyl group, Hecke algebra, and Schur algebra of type $\tilde{A}_{n-1}$. The extended affine Weyl group may be defined in multiple ways. We shall give the two that are relevant for our arguments.

\begin{definition}
    Let $n\in\z_{>0}$. The extended affine Weyl group of type $\tilde{A}_{n-1}$ is the group $\weyl$ of bijections $x:\z\rightarrow\z$ such that
    \begin{itemize}
        \item $x(i)+n=x(i)+n$ for all $i\in\z$, and
        \item $\sum_{i=1}^n\left(x\left(i\right)-i\right)\equiv0 \ \left(\text{mod} \ n\right)$
    \end{itemize}
    
    For $0\leq i\leq n$, let $s_i:\z\rightarrow\z$ be the map
    \begin{align*}
        s_i(j)=\begin{cases}
            j+1 \ &\text{if} \ j\equiv i \ \left(\text{mod} \ n\right) \\
            j-1 \ &\text{if} \ j\equiv i+1 \ \left(\text{mod} \ n\right) \\
            j \ &\text{if} \ j\not\equiv i,i+1 \ \left(\text{mod} \ n\right) \\
        \end{cases}
    \end{align*}
    Note that $s_i\in\weyl$ if and only if $n>1$, and that $s_0=s_n$.
    
    Let $\omega\in\weyl$ be the map \begin{align*}
        \omega(j)=j+1.
    \end{align*}
    
    Let $\affw$ be the subgroup of $\weyl$ generated by $\affgen=\left\{s_i\middle|0\leq i\leq n-1\right\}\cap\weyl$ and let $\Omega$ be the subgroup of $\weyl$ generated by $\omega$.
\end{definition}

\begin{prop}\label{weyl group affine decomposition} There is a second description for $\weyl$:
\begin{enumerate}
    \item $\affw$ is the affine Weyl group of type $\tilde{A}_{n-1}$. For $n>2$, this is the Coxeter group with generators $\affgen$ and relations
\begin{align*}
    s_is_{i+1}s_{i}&=s_{j}s_{i+1}s_{j} \ \quad \text{for} \ 0\leq i\leq n-1 \\
    s_is_j&=s_js_i \qquad \text{for} \ 0\leq i< j\leq n-1, \ j-i\notin \left\{1, n-1\right\} \\
    s_i^2&=1 \qquad \quad \text{for} \ 0\leq i\leq n-1.
\end{align*}
For $n=2$, this is instead the Coxeter group with the same generators but only the relation
\begin{align*}
        s_i^2&=1 \quad \text{for} \ 0\leq i\leq 1.
\end{align*}
For $n=1$, this is the trivial Coxeter group.
\item $\Omega\cong\z$.
\item $\omega s_{i}=s_{i+1}\omega$ for $0\leq i\leq n-1$.
\item $\weyl=\affw\rtimes\Omega$.
\item $\omega^n$ is central.
\end{enumerate}
\end{prop}

\begin{proof}
 The first, second and fourth claim are \cite[Theorem 2.1.3(c)]{xi2002based}, while the third and final claims are \cite[Theorem 2.1.3(b)]{xi2002based} and \cite[Theorem 2.1.3(a)]{xi2002based} respectively.
\end{proof}

We shall need the length function and the Bruhat order on $\weyl$ to construct our cellular structure.

\begin{definition}
A reduced expression for $w'\in\affw$ is a sequence of elements of $\affgen$ whose product is $w'$, and which contains a minimal number of $s\in\affgen$ among all such sequences.

The length $l(w)$ of $w=w'\omega^i$ is the number of $s\in\affgen$ occurring in some (hence any) reduced expression for $w'$.

If $v=v'\omega^j$ for $v'\in\affw$, we write $v\leq w$ precisely when both $i=j$ and any reduced expression for $w'$ has a subsequence that is a reduced expression for $v'$. This defines a partial order on $\weyl$.
\end{definition}

%Indeed, the Bruhat order is also defined by a second, weaker condition.

%\begin{prop}\label{existence enough for bruhat}
    %If $v=v'\omega^j$ for $v'\in\affw$, then $v\leq w$ precisely when both $i=j$ and some reduced expression for $w'$ has a subsequence that is a reduced expression for $v'$.
%\end{prop}%hecke algebras with unequal parameters 2.4

As we are in type $A_{n-1}$, the length function has an explicit formula, which we will also need.

\begin{prop}\label{Formula for length}
    $l(w)=\sum_{1\leq i<j\leq n}\left|\left\lfloor\frac{w(j)-w(i)}{n}\right\rfloor\right|$.
\end{prop}

\begin{proof}
 This is \cite[Theorem 2.1.3(e)]{xi2002based}.
\end{proof}

Using the length function, we can define the Hecke algebra using the Iwahori-Matsumoto presentation.

\begin{definition}\label{Hecke algebra definition}
Let $\zq=\z\left[q\right]$, let $\zpmq=\z\left[q^{\pm 1}\right]$, and let $\zrtq=\z\left[q^{\pm\frac{1}{2}}\right]$, viewing each as an algebra over the previous by identifying the parameters $q$. Where we want to work simultaneously over all three rings, we shall write $\zany$.

The extended affine Hecke algebra of type $\tilde{A}_{n-1}$ is the $\zany$-algebra $\hecany$ with generators $\left\{T_{w}\middle|w\in\weyl\right\}$ subject to the relations:
\begin{align}\label{hecke algebra relations}
\begin{split}
        T_{v}T_{w}=T_{vw} \quad &\text{if} \ l(v)+l(w)=l(vw) \\
    T_{s}^2=(q-1)T_{s}+q \quad &\text{for} \ s\in\affgen
\end{split}
\end{align}

%Let $\hecpmq=\zpmq\otimes_{\zq}\hecq$ and $\hecrtq=\zrtq\otimes_{\zpmq}\hecpmq=\zrtq\otimes_{\zq}\hecq$, which are respectively $\zpmq$- and $\zrtq$-algebras. Write $\hecany$ for any of $\hecq$, $\hecpmq$, and $\hecrtq$.

Let $\affhecany$ be the subalgebra of $\hecany$ generated by $\left\{T_{w}\middle|w\in\affw\right\}$.
\end{definition}

Like the Weyl group, the Hecke algebra can also be decomposed into affine and cyclic parts.

\begin{prop}\label{reducing to affine hecke algebra} We can describe $\hecany$ in terms of $\affhecany$ and $\Omega$:
\begin{enumerate}
    \item $\Omega$ is a multiplicative submonoid of $\hecany$ via the inclusion $\omega\mapsto T_{\omega}$.
    \item $T_{w\omega}=T_{w}\omega$ and $T_{\omega w}=\omega T_{w}$ for $w\in\weyl$.
    \item $\hecany=\affhecany\rtimes\Omega$.
    \item $\omega^n$ is central.
        \item $\affhecany$ is the $\zany$-algebra with generators $\left\{T_{w}\middle|w\in\affw\right\}$ subject to the relations \cref{hecke algebra relations}.
    \item $\hecany$ and $\affhecany$ are free $\zany$-modules with bases $\left\{T_{w}\middle|w\in\weyl\right\}$ and \\ $\left\{T_{w}\middle|w\in\affw\right\}$ respectively.
\end{enumerate}
\end{prop}

\begin{proof}
 The first and second claims follow immediately from \cref{hecke algebra relations} since $l(\omega)=0$. Together with \cref{weyl group affine decomposition} (4) and (5) these then give the third and fourth claims respectively. The second claim and \cref{weyl group affine decomposition} (4) also allow us to simplify the first relation to $T_{v}T_{\omega^iw\omega^{-i}}=T_{v\omega^iw\omega^{-i}}$ for $v,w\in\affw$ with $l(v)+l(w)=l(vw)$. But $l(w)=l(\omega w\omega^{-1})$ by \cref{weyl group affine decomposition} (3), so we may reduce this further to $T_{v}T_{w}=T_{vw}$ for $v,w\in\affw$ with $l(v)+l(w)=l(vw)$. Hence we have the fifth claim. The sixth claim for $\affhecany$ is then \cite[Proposition 3.3]{lusztig2003hecke}, and the sixth claim for $\hecany$ then follows by the third claim and \cref{weyl group affine decomposition} (4).
\end{proof}

The we shall give the definition of Schur algebra using certain elements of Hecke algebra related to parabolic subgroups of $\weyl$. We first recall the necessary definitions and properties of parabolic subgroups.

\begin{definition}
Let $\fingen=\affgen\backslash\left\{s_0\right\}$.

For $P\subseteq \fingen$, let $\weyl_P$ be the subgroup of $\weyl$ generated by $P$.
\end{definition}

\begin{prop}
    $\weyl_P$ is finite.
\end{prop}

\begin{proof}
 This follows form \cite[Section 1.20]{lusztig2003hecke}. Indeed, it is easy to see directly that $\weyl_P$ is a product of finite symmetric groups.
\end{proof}

Now we are ready to define the Schur algebra.

\begin{definition}
Let $x_P=\sum_{p\in \weyl_P}T_p$.

The affine $q$-Schur algebra is the $\zany$-algebra
\begin{align*}
    \schany=\End_{\hecany}\left(\bigoplus_{P\subseteq\fingen}x_P\hecany\right)=\bigoplus_{P,Q\subseteq\fingen}\Hom_{\hecany}\left(x_P\hecany,x_Q\hecany\right).
\end{align*}
\end{definition}

We conclude this section by giving the standard basis of the Schur algebra in terms of longest double coset representatives.

\begin{prop}\label{maximal length representative}
 For $P,Q\subseteq \fingen$ and $W_QwW_P\in\weyl_Q\backslash \weyl/ \weyl_P$, there is a unique $d\in W_QwW_P$ of maximal length. In particular, $\weyl_P$ has a unique element $w_P$ of maximal length.
\end{prop}

\begin{proof}
 This is \cite[Proposition 9.15(e)]{lusztig2003hecke}.
\end{proof}

\begin{definition}
For $P,Q\subseteq \fingen$, let $\prescript{Q}{}\weyl^{P}$ denote the set of maximal length representatives for the double cosets $\weyl_Q\backslash \weyl/ \weyl_P$.

For $w\in\prescript{Q}{}\weyl^{P}$ define $T_{QP}^w\in\Hom_{\hecany}\left(x_P\hecany,x_Q\hecany\right)\subseteq\schany$ by 
\begin{align*}
    x_Ph\mapsto \left(\sum_{w'\in\weyl_Qw\weyl_P}T_{w'}\right)h.
\end{align*}
\end{definition}

\begin{prop}\label{schur algebra basis} $\schany$ admits a nice basis extending that of $\hecany$:
\begin{enumerate} 
    \item $\schany$ is a free $\zany$-module with basis $\left\{T_{QP}^w\middle|P,Q\subseteq \fingen, w\in \prescript{Q}{}\weyl^{P} \right\}$.
    \item $W_{\emptyset}=\left\{1\right\}$. Hence, $\hecany$ is a subalgebra of $\schany$ via $T_{w}\mapsto T_{\emptyset\emptyset}^w$.
\end{enumerate}
\end{prop}

\begin{proof}
The first claim is \cite[Theorem 2.2.4]{green1999affine}. The second claim is immediate.
\end{proof}

\section{The Kazhdan-Lusztig basis and cells}

In this section we define the cells of the Hecke and Schur algebras, which form the first piece of the affine cellular structure. We define these as spans of subsets of the Kazhdan-Lusztig basis; hence we must first define this basis. The key ingredient is the bar involution, whose definition we now give.

\begin{prop}\label{bar involution}
Suppose $q^{-1}\in\zany$. Then

\begin{enumerate}
    \item $T_w$ is invertible.
    \item The $\z$-linear map $\hecany\rightarrow\hecany$ given by $q^{i}\mapsto q^{-i}$ and $T_{w}\mapsto T_{w^{-1}}^{-1}$ is an anti-involution.
\end{enumerate}
\end{prop}

\begin{proof}
 For $w\in\affw$, the first claim is \cite[Section 3.5]{lusztig2003hecke} and the second claim is \cite[Lemma 4.2]{lusztig2003hecke}. The claims for general $w\in\weyl$ then follow from \cref{reducing to affine hecke algebra} (1) and (2).
\end{proof}

Using this involution, we can define the Kazhdan-Lusztig polynomials, which in turn define the Kazhdan-Lusztig basis of the Hecke algebra.

\begin{prop}\label{kl basis}
There exists a unique set of elements
    \begin{align*}
        \left\{P_{y,w}\in\zq\middle|y,w\in\weyl, y\leq w\right\}
    \end{align*}
such that 
\begin{itemize}
    \item $P_{w,w}=1$ for all $w\in\weyl$,
    \item $\deg_q(P_{y,w})\leq \frac{1}{2}\left(l(w)-l(y)-1\right)$ whenever $y\neq w$
    \item The elements $\left\{\kln_w\middle|w\in\weyl\right\}$ defined by
    \begin{align*}
        \kln_w=q^{-\frac{l(w)}{2}}\sum_{y\leq w}P_{y,w}T_y
    \end{align*}
     are each invariant under the anti-involution of \cref{bar involution}.
\end{itemize}
Furthermore,
\begin{enumerate}
    \item $P_{y'\omega^j,w'\omega^{i}}=P_{y',w'}\delta_{j,i}$ for $y'w'\in\affw$. In particular, $\kln_{w'}\in\affhecrtq$.
    \item $\kln_{w'\omega^{i}}=\kln_{w'}\omega^{i}$ and $\kln_{\omega^iw'}=\omega^i\kln_{w'}$
    \item $\left\{\kln_w\middle|w\in\weyl\right\}$ is a basis of $\hecrtq$.
\end{enumerate}

\end{prop}

\begin{proof}
 If $w'\in\affw$ and $P_{y,w'}\neq0$ then $y'\in\affw$. Hence that there exist unique $P_{y,w'}$ for $w'\in\affw$ satisfying the listed properties, and such that $\left\{\kln_{w'}\middle|w'\in\affw\right\}$ is a basis of $\affhecrtq$, follows from \cite[Theorem 1.1]{kazhdan1979representations}. Defining $P_{y,w'\omega^{i}}=P_{y\omega^{-i},w'}$ gives by \cref{reducing to affine hecke algebra} (2) that $C_{w'\omega^{i}}=C_{w'}\omega^{i}$, and hence that the $P_{y,w}$ satisfy the listed properties. The first claim follows by definition, and the third claim follows by \cref{reducing to affine hecke algebra} (3).
 
 To see that the $P_{y,w}$ are unique, observe that for any choice $P_{y,w}'$ satisfying the listed properties, we have for $w'\in\affw$ and by \cref{reducing to affine hecke algebra} (2) that $C_{w'\omega^{i}}'\omega^{-i}=\sum_{y'\leq w'}P_{y'\omega^i,w'\omega^i}'T_{y'}$, Since this is invariant under the anti-involution of \cref{bar involution} by assumption, writing $P_{y',w'}''=P_{y'\omega^i,w'\omega^i}'$ we get that \\ $\left\{P_{y',w'}''\in\zq\middle|w'\in\affw, y'\leq w'\right\}$ satisfies the listed properties, and so by uniqueness of the $P_{y,w'}$ for $w'\in\affw$ we must have that $P_{y'\omega^i,w'\omega^i}'=P_{y',w'}''=P_{y',w'}=P_{y'\omega^i,w'\omega^i}$.
 
To see that $\kln_{\omega^iw'}=\omega^i\kln_{w'}$, observe that $P_{\omega^iy,\omega^{i}w'}=P_{\omega^iy\omega^{-i},\omega^iw'\omega^{-i}}$, and so by \cref{reducing to affine hecke algebra} (2) we are reduced to showing that $P_{\omega y'\omega^{-1},\omega w'\omega^{-1}}=P_{y',w'}$. But by \cref{reducing to affine hecke algebra} (3) the conjugation action of $\omega$ is an automorphism of $\affhecrtq$. Furthermore, by \cref{weyl group affine decomposition} (3) it is length-preserving, and by \cref{reducing to affine hecke algebra} (2) it commutes with the anti-involution of \cref{bar involution}, so the $P_{\omega y'\omega^{-1},\omega w'\omega^{-1}}$ satisfy all three conditions above, and hence by uniqueness of the $P_{y',w'}$ we have $P_{\omega y'\omega^{-1},\omega w'\omega^{-1}}=P_{y',w'}$.
\end{proof}

We also use the Kazhdan-Lusztig polynomials to define the Kazhdan-Lusztig basis for the Schur algebra, as well as a renormalised variant of the basis which is defined over $\zq$. This latter basis will be how we define the cell structure over $\zq$.

\begin{definition}
For $w\in\weyl$, let $\klu_w=q^{\frac{l(w)}{2}}\kln_w=\sum_{y\leq w}P_{y,w}T_y$.

For $P,Q\subseteq\fingen$ and $w\in \prescript{Q}{}\weyl^{P} $, let $\klu_{QP}^w=\sum_{y\leq w}P_{y,w}T_{QP}^{y}$ and let $\kln_{QP}^w=q^{\frac{l(w_P)-l(w)}{2}}\klu_{QP}^w$.
\end{definition}

\begin{lemma}These give bases:
\begin{enumerate}
    %\item $\left\{\kln_w\middle|w\in\weyl\right\}$ is a basis of $\hecrtq$.
    \item $\left\{\klu_w\middle|w\in\weyl\right\}$ is a basis of $\hecany$.
    \item $\left\{\klu_{QP}^w\middle| P,Q\subseteq\fingen, w\in \prescript{Q}{}\weyl^{P} \right\}$ and \\ $\left\{\kln_{QP}^w\middle| P,Q\subseteq\fingen, w\in \prescript{Q}{}\weyl^{P} \right\}$ are bases of $\schany$ and $\schrtq$ respectively.
    \item The inclusion $\hecany\rightarrow\schany$ sends $\klu_w\mapsto\klu_{\emptyset\emptyset}^w$
    \item The inclusion $\hecrtq\rightarrow\schrtq$ sends $\kln_w\mapsto\kln_{\emptyset\emptyset}^w$.
\end{enumerate}
\end{lemma}

\begin{proof}
 By \cref{reducing to affine hecke algebra} (6) and \cref{kl basis}, the set $\left\{\klu_w\middle|w\in\weyl\right\}$ is related to the basis $\left\{T_w\middle|w\in\weyl\right\}$ by a matrix in $\zq$ that is triangular with respect to $\leq$ and has $1$ on the diagonal. Thus we have the first claim. That $\left\{\klu_{QP}^w\middle| P,Q\subseteq\fingen, w\in \prescript{Q}{}\weyl^{P} \right\}$ is a basis follows by the same logic, as again by \cref{schur algebra basis} (1) and \cref{kl basis} it is related to the basis \\ $\left\{T_{QP}^w\middle|P,Q\subseteq \fingen, w\in \prescript{Q}{}\weyl^{P} \right\}$ by a matrix in $\zq$ that is triangular with respect to $\leq$ and has $1$ on the diagonal. Then, as $\kln_{QP}^w=q^{\frac{l(w_P)-l(w)}{2}}\klu_{QP}^w$, we immediately get the rest of the second claim. The last two claims are immediate by \cref{schur algebra basis} (2).
\end{proof}

We also make some observations about the structure coefficients of the Kazhdan-Lusztig basis of the Hecke algebra, which we will need later.

\begin{definition}
    We write $\scn_{u,v}^w$ for the coefficient of $\kln_w$ in the product $\kln_u\kln_v$. Similarly, write $\scu_{u,v}^w$ for the coefficient of $\klu_w$ in the product $\klu_u\klu_v$.
\end{definition}
Thus, $\scu_{u,v}^w=q^{\frac{l(u)+l(v)-l(w)}{2}}\scn_{u,v}^w$. Note that $\scu_{u,v}^w\in\zq$, but in general we only have $\scn_{u,v}^w\in\zrtq$.

\begin{lemma}\label{structure constant facts}
Let $u',v'\in\affw$. Then $\scn_{u'\omega^i ,v'\omega^j}^{w}=0$ unless $w=w'\omega^{i+j}$ for $w'\in\affw$, and in this case $\scn_{u'\omega^i ,v'\omega^j}^{w}=\scn_{u',\omega^i v'\omega^{-i} }^{w'}$.

Furthermore, $\scn_{\omega^{i}u'\omega^{-i} ,\omega^{i}v'\omega^{-i}}^{\omega^{i}w''\omega^{-i}}=\scn_{u',v' }^{w'}$.

Finally, these equations also hold with $\scn$ replaced with $\scu$.
\end{lemma}

\begin{proof}
 The first part follows from \cref{kl basis} (1) and (2) and \cref{reducing to affine hecke algebra} (3). The second follows by \cref{kl basis} (2). The final claim follows by the same logic, or by noting that $\scu_{u,v}^w=q^{\frac{l(u)+l(v)-l(w)}{2}}\scn_{u,v}^w$ and that multiplication by $\omega$ does not change the lengths of elements of $\weyl$.
\end{proof}

We also want to make similar observations about the structure coefficients of the Kazhdan-Lusztig basis of the Schur algebra. For this, we need the Poincar\'e polynomial of a parabolic subgroup of $\weyl$.

\begin{definition}
    Let $P\subseteq\fingen$. The Poincare polynomial of $P$ is $p_P=\sum_{w\in W_P}q^{l(w)}$.
\end{definition}

We also need the following proposition, which explains our choice to use maximal coset representatives.

\begin{prop}\label{kl basis and idempotents}
    $w$ is a maximal length representative of $W_QwW_P$ if and only if $T_s\klu_w=q\klu_w$ and $\klu_wT_t=q\klu_w$ for all $s\in Q$ and $t\in P$.

The basis element $\klu_{QP}^w$ is exactly the map $x_Ph\mapsto \klu_{w}h$. 
\end{prop}

\begin{proof}
By \cite[Lemma 9.8]{lusztig2003hecke}, the element $w$ is maximal length in $W_QwW_P$ precisely when $sw<w$ for all $s\in Q$ and $ws<w$ for all $s\in P$. By \cite[Theorem 6.6]{lusztig2003hecke} we have that $sw<w$ if and only if $T_s\kln_w=q\kln_w$, and $ws<w$ if and only if $\kln_wT_s=q\kln_w$.

Since $sw<w$ for all $s\in Q$, we have by \cite[Theorem 6.6]{lusztig2003hecke} again that if $y<sy\leq w$ then $P_{y,w}=P_{sy,w}$. Similarly, if $t\in P$ and $y<yt\leq w$ then $P_{y,w}=P_{yt,w}$. Hence if $y\leq w$ and $x\in W_QyW_P$ then $P_{y,w}=P_{x,w}$. Therefore $\klu_{QP}^w(x_P)=\sum_{y\leq w}P_{y,w}T_{QP}^y(x_P)=\sum_{x\leq w}P_{x,w}T_{x}=\klu_{w}$.

%To see the second claim, note that by the first claim we have $\klu_w x_P=p_P\klu_w$. Thus the map $x_Ph\mapsto \klu_w x_P h$ is in fact equal to the map $x_Ph\mapsto p_P\klu_w h$. Hence if $x_Ph=x_Ph'$ then $p_P\klu_w h=\klu_w x_P h=\klu_w x_P h'=p_P\klu_w h'$, and so as $\hecany$ is free and $\zq$ is a domain we have $\klu_{w}h=\klu_wh'$, and so $x_Ph\mapsto \klu_{w}h$ is a well-defined map. Now, by the first claim again, we have  $x_Q\klu_w =p_Q\klu_w$, and so $p_Q\klu_w\in x_Q\hecany$. 

%Thus, the map $x_Ph\mapsto p_Q\klu_wh$ must be a linear combination of $T_{QP}^x$, say with coefficients $\alpha_x$. But $T_{QP}^x(x_P)$ is a sum of $T_y$ for $y\in W_QxW_P$, and $p_Q\klu_w$ is a linear combination of $T_y$ for $y\leq w$, so $\alpha_x$ is nonzero only for those $x$ for which some $y\in W_QxW_P$ has $y\leq w$. But then all $T_y$ with $y\in W_QxW_P$ must occur in $p_Q\klu_w$ with coefficient $\alpha_x$, so in fact $\alpha_{x}=p_QP_{x,w}$, and hence $x_Ph\mapsto p_Q\klu_wh$ is exactly the map $p_Q\klu_{QP}^w$.
\end{proof}

%Note that this proof also show that, if $w$ is maximal length in $W_QwW_P$, then $P_{y,w}=P_{x,w}$ for all $y\in W_QxW_P$.

Now we can give the Schur algebra analog of \cref{structure constant facts}.

\begin{lemma}\label{schur algebra structure coefficients}
The coefficient $\scu_{uRv}^{QwP}$ of $\klu_{QP}^w$ in the product $\klu_{QR}^u\klu_{RP}^v$ is $p_{R}^{-1}\scu_{u,v}^w$, and the coefficient $\scn_{uRv}^{QwP}$ of $\kln_{QP}^w$ in the product $\kln_{QR}^u\kln_{RP}^v$ is $p_{R}^{-1}q^{\frac{l(w_R)}{2}}\scn_{u,v}^w$.
\end{lemma}

\begin{proof}
 By \cref{kl basis and idempotents} the map $\klu_{QR}^ux_R\klu_{RP}^v$ sends $x_Ph$ to $\klu_u\klu_vh=\scu_{u,v}^w\klu_wh$, so  $\klu_{QR}^ux_R\klu_{RP}^v=\sum_{w\in\weyl}\scu_{u,v}^w\klu_{QP}^w$. But again by \cref{kl basis and idempotents} we have that $x_R\klu_{RP}^v$ sends $x_Ph$ to $x_R\klu_v h=p_R\klu_v h$, so $x_R\klu_{RP}^v=p_R\klu_{RP}^v$.
 
 The second claim follows from the first since $\scu_{u,v}^w=q^{\frac{l(u)+l(v)-l(w)}{2}}\scn_{u,v}^w$ and $\kln_{QP}^w=q^{\frac{l(w_P)-l(w)}{2}}\klu_{QP}^w$.
\end{proof}

Before proceeding further, we recall the cellular involution of the Hecke and Schur algebras. This is both necessary for the affine cellular structure and will simplify the following exposition.

\begin{prop}\label{cellular involution}
The $\zany$-linear map $\iota:\schany\rightarrow\schany$ given by $\klu_{QP}^w\mapsto\klu_{PQ}^{w^{-1}}$ is an anti-involution of $\schany$. It restricts to an anti-involution $\klu_w\mapsto\klu_{w^{-1}}$ of $\hecany$.
\end{prop}

\begin{proof}
 This is \cite[Equation 1.7(a)]{lusztig1999aperiodicity}. Alternatively, the claim for $\hecany$ and $w\in\affw$ is \cite[Section 5.6]{lusztig2003hecke}, and its extension to all $w\in\weyl$ follows from \cref{kl basis} (2). The claim for $\schany$ then follows by \cref{schur algebra structure coefficients}.
\end{proof}

We define cells as subsets of the Kazhdan-Lusztig basis. These are equivalence classes which are in turn defined from various preorders. We now present the definition of these preorders.

\begin{definition}
    An ideal (left, right, or two-sided) in $\hecany$ is called based if it is a free module with a basis of elements of the form $\klu_w$. Similarly, an ideal (left, right, or two-sided) in $\schany$ is called based if it has a basis of elements of the form $\klu_{QP}^w$.

    We write $v\leq_Lw$ (respectively $v\leq_Rw$, respectively $v\leq_{LR}w$) on $v,w\in\weyl$ if $\klu_{u}$ is contained in the based left (respectively right, respectively two-sided) ideal of $\hecany$ generated by $\klu_{w}$.

    Similarly, we write $\klu_{Q'P'}^v\leq_L\klu_{QP}^w$ (respectively $\klu_{Q'P'}^v\leq_R\klu_{QP}^w$, respectively $\klu_{Q'P'}^v\leq_{LR}\klu_{QP}^w$) if $\klu_{Q'P'}^v$ is contained in the based left (respectively right, respectively two-sided) ideal of $\schany$ generated by $\klu_{QP}^w$.
\end{definition}

Observe that these are preorders, that they do not depend on the choice of $\zany$, and that, for $\zrtq$, using $\kln$ instead of $\klu$ gives the same order. Furthermore, $\leq_{LR}$ is the join of $\leq_{L}$ and $\leq_{R}$, and, by applying the involution $\iota$, we have $v\leq_L w$ if and only if $v^{-1}\leq_{R}w^{-1}$, and $\klu_{Q'P'}^v\leq_L\klu_{QP}^w$ if and only if $\klu_{P'Q'}^{v^{-1}}\leq_R\klu_{PQ}^{w^{-1}}$. We also make the following observation on how these behave on the affine subalgebra.

\begin{lemma}\label{cells on the affine subgroup}
    We have for $v',w'\in\affw$ that $\omega^iv'\leq_L \omega^jw'$ (respectively $v'\omega^i\leq_R w'\omega^j$, respectively $\omega^iv'\omega^{i'}\leq_{LR} \omega^jw'\omega^{j'}$) if and only if $v'\leq_L w'$ (respectively $v'\leq_R w'$, respectively $v'\leq_{LR}w'$).
    
    Define preorders $\leq_{L}'$, $\leq_{R}'$, and $\leq_{LR}'$ on $\affw$ using based ideals in $\affhecany$ instead of $\hecany$. Then these are the restrictions to $\affw$ of  $\leq_L$, $\leq_R$, and $\leq_{LR}$ respectively.
\end{lemma}

\begin{proof}
The first claim is immediate from \cref{reducing to affine hecke algebra} (2). The second follows immediately for $\leq_L$ and $\leq_R$ by \cref{reducing to affine hecke algebra} (2) and (3). For $\leq_{LR}$ however more work is required: we must show that $\omega^{i}w\omega^{-i}\leq_{LR}'w$ for $i=\pm1$. This can be seen by the explicit description of cells in \cite[Theorem 17.4]{shi1986kazhdan}.
\end{proof}

Now we can define the equivalence relations which will give the cells.

\begin{definition}
        We write $\sim_L$ (respectively $\sim_R$, respectively $\sim_{LR}$) for the equivalence relation given by the meet of $\leq_L$ and $\geq_L$ (respectively, $\leq_R$ and $\geq_R$, respectively $\leq_{LR}$ and $\geq_{LR}$).
\end{definition}

The following easy lemma and will be of technical importance later.

\begin{lemma}\label{omega to the n}
    $\klu_w\sim_{L}\omega^n \klu_w\sim_R \klu_w$ for any $\klu_w$.
\end{lemma}

\begin{proof}
 This is immediate from \cref{reducing to affine hecke algebra} (4).
\end{proof}

We can also relate the equivalence relations on the Schur algebra to those on the Hecke algebra.

\begin{prop}\label{hecke and shcur cell relations}
The equivalence relations on $\hecany$ and $\schany$ are related:
\begin{enumerate}
    \item $\klu_{QP}^w\sim_L\klu_{Q'P'}^v$ if and only if $P=P'$ and $w\sim_{L} v$.
    \item $\klu_{QP}^w\sim_R\klu_{Q'P'}^v$ if and only if $Q=Q'$ and $w\sim_{R} v$.
    \item $\klu_{QP}^w\sim_{LR}\klu_{Q'P'}^v$ if and only if $w\sim_{LR} v$.
\end{enumerate}
Thus we have a correspondence between left cells in $\schany$ and pairs $(\Gamma,P)$, where $\Gamma$ is a left cell in $\hecany$ and $P\subseteq\fingen$, and similarly for right and two-sided cells.
\end{prop}

\begin{proof}
 This is \cite[Proposition 3.8]{mcgerty2003cells}.
\end{proof}

We now finally have everything we need to define cells.

\begin{definition}
        Fix $C$ some $\sim_L$-equivalence class (respectively $\sim_R$-equivalence class, respectively $\sim_{LR}$-equivalence class) in $\weyl$. A left cell (respectively right cell, respectively two-sided cell) in $\hecany$ is the $\zany$-linear span of $\klu_w$ for all $w\in C$.
        
        Similarly, fix $D$ some $\sim_L$-equivalence class (respectively $\sim_R$-equivalence class, respectively $\sim_{LR}$-equivalence class) of the $\klu_{QP}^w$. A left cell (respectively right cell, respectively two-sided cell) in $\schany$ is the $\zany$-linear span of $\klu_{QP}^w$ for all $\klu_{QP}^w\in D$.
        
        The left cells (respectively right cells, respectively two-sided cells) inherit a partial order from $\leq_L$ (respectively $\leq_R$, respectively $\leq_{LR}$).
\end{definition}

Observe that the anti-involution $\iota$ sends left cells to right cells, and conversely. We can see immediately from the definition that $\hecany$ and $\schany$ are both the direct sum of their two-sided cells. Furthermore, an exact indexing of the two-sided cells is known in terms of combinatorial data, which we now give.

\begin{definition}
        A partition of $n$ is a non-increasing sequence $\lambda$ of positive integers with sun $n$. We write $\lambda\vdash n$, and write $\Lambda$ for the set of all partitions of $n$.
        
        Given $\lambda,\mu\vdash n$, we say $\lambda\leq\mu$ if, for all $i$, we have $\sum_{i'=1}^i\lambda_{i'}\leq\sum_{i'=1}^i\mu_{i'}$. We call this order the dominance order.
\end{definition}

\begin{prop}\label{index for two sided cells}
There are isomorphisms of partially ordered sets between $\left(\Lambda,\geq\right)$, the set of two-sided cells of $\hecany$ ,and the set of two-sided cells of $\schany$, where both the latter two are ordered by $\leq_{LR}$ (note the reversal of the order). In particular, there are finitely many two-sided cells.
\end{prop}

\begin{proof}
 For $\affhecany$, this is \cite[Section 2.9]{shi1996partial}. Hence it follows for $\hecany$ by \cref{cells on the affine subgroup} and for $\schany$ by \cref{hecke and shcur cell relations}.
\end{proof}

\begin{definition}
        We write $\hecany^{\lambda}$ and $\schany^{\lambda}$ for the two-sided cells of $\hecany$ and $\schany$ respectively that correspond under the isomorphism \cref{index for two sided cells} to the partition $\lambda$.
        
        Similarly, we write $\hecany^{\geq\lambda}=\bigoplus_{\mu\geq\lambda}\hecany^{\mu}$ and $\schany^{\geq\lambda}=\bigoplus_{\mu\geq\lambda}\schany^{\mu}$, and furthermore we write $\hecany^{>\lambda}=\bigoplus_{\mu>\lambda}\hecany^{\mu}$ and $\schany^{>\lambda}=\bigoplus_{\mu>\lambda}\schany^{\mu}$.
\end{definition}

We will need some elementary properties of the cells. Firstly, we observe how they behave under the involution.

\begin{prop}\label{cells respect involution}
    $\iota\left(\hecany^{\lambda}\right)=\hecany^{\lambda}$ and $\iota\left(\schany^{\lambda}\right)=\schany^{\lambda}$. In particular, $\iota$ sends left cells in $\hecany^{\lambda}$ and $\schany^{\lambda}$ to right cells in $\hecany^{\lambda}$ and $\schany^{\lambda}$ respectively, and conversely.
\end{prop}

\begin{proof}
 This is \cite[Theorem 1.10]{lusztig1987cells}.
\end{proof}

Next, we observe that the various unions of cells are, by construction, ideals in suitable quotients.

\begin{lemma}\label{cells are ideals}
    $\hecany^{\geq\lambda}$ and $\hecany^{>\lambda}$ are two-sided ideals of $\hecany$, on which $\iota$ restricts to a well-defined anti-involution. Similarly, $\schany^{\geq\lambda}$ and $\schany^{>\lambda}$) are two-sided ideals of $\schany$ on which $\iota$ restricts to a well-defined anti-involution. Thus, $\hecany^{\lambda}=\hecany^{\geq\lambda}/\hecany^{>\lambda}$ and $\schany^{\lambda}=\schany^{\geq\lambda}/\schany^{>\lambda}$ inherit algebra structures as subquotients of $\hecany$ and $\schany$ respectively, on which $\iota$ again restricts to a well-defined anti-involution.
\end{lemma}

\begin{proof}
 That the various submodules are ideals is immediate from their definitions. That the involution restricts to these ideals follows from \cref{cells respect involution}.
\end{proof}

\begin{lemma}\label{left cells are left ideals}
    Any left (respectively right) cell in $\hecany^{\lambda}$ is a left (respectively right) ideal in $\hecany/\hecany^{>\lambda}$. Similarly, any left (respectively right) cell in $\schany^{\lambda}$ is a left (respectively right) ideal in $\schany/\schany^{>\lambda}$.
\end{lemma}

\begin{proof}
 This is again immediate from the definitions.
\end{proof}

Finally, we recall that the number of left (and hence right) cells in each two-sided cell is finite, and furthermore that the exact number is known.

\begin{prop}\label{number of left cells}
    The number of left (equivalently, right) cells in $\hecany^{\lambda}$ is $n_{\lambda}=\frac{n!}{\mu_1!\dots\mu_{r'}!}$ where $(\mu_1,\dots,\mu_{r'})$ is the dual partition of $\lambda$, that is, $\mu_i=\left|\left\{j\middle| \lambda_j\geq i\right\}\right|$. The number of left (equivalently right) cells in $\schany^{\lambda}$ is $m_{\lambda}=\prod_{i=1}^{n-1}\left(\frac{n!}{(n-i)!i!}\right)^{n_i}$, where $n_i=\lambda_i-\lambda_{i+1}$.
\end{prop}

\begin{proof}
By \cite[Theorem 14.4.5]{shi1986kazhdan}, ${\affhecany}^{\lambda}$ has exactly $n_{\lambda}$ left cells. Hence by \cref{cells on the affine subgroup} so does $\hecany^{\lambda}$. That $\schany^{\lambda}$ has $m_{\lambda}$ left cells is \cite[Proposition 4.10]{mcgerty2003cells}.
\end{proof}

\section{Integral isomorphisms between left cells}

In this section we give an isomorphism between each pair of left cells in the same two-sided cell. We construct this isomorphism using the Kazhdan-Lusztig star operation on $\weyl$. We first recall its definition.

\begin{definition}
        For $w\in\weyl$, write $L(w)=\left\{s\in\affgen\middle|sw<w\right\}$ and $R(w)=\left\{s\in\affgen\middle|ws<w\right\}$.
        
        Suppose $n\geq 3$. For $s_i\in\affgen$, let $D_{L}(s_i)=\left\{w\in\weyl\middle|\ \left|L(w)\cap\left\{s_i,s_{i+1}\right\}\right|=1\right\}$ and $D_{R}(s_i)=\left\{w\in\weyl\middle|\ \left|R(w)\cap\left\{s_i,s_{i+1}\right\}\right|=1\right\}$.
\end{definition}

\begin{prop}
    If $w\in D_L(s_i)$, then $\left|\left\{s_iw, s_{i+1}w\right\}\cap D_{L}(s_i)\right|=1$. Similarly, if $w\in D_R(s_i)$, then $\left|\left\{ws_i, ws_{i+1}\right\}\cap D_{R}(s_i)\right|=1$.
\end{prop}

\begin{proof}

This is \cite[Section 1.6]{shi1986kazhdan}.
 %First, by \cite[Section 2.1, Proposition 2.4]{lusztig2003hecke} we have $sw<w$ if and only $l(sw)=l(w)-1$ and $sw>w$ if and only if $l(w)+1$. Furthermore, by \cite[Lemma 1.2]{lusztig2003hecke} these are the only possibilities. Without loss of generality, suppose $s_iw<w$.  Then, as $w\in D_L(s_i)$ we must have $s_{i+1}w> w$. Write $u=s_iw$. Then $s_{i+1}s_iu>s_iu>u$.
 
 %First, suppose $v=s_{i+1}u<u$. Then $u=s_iw\in D_L(s_i)$ and furthermore $s_{i+1}w=s_{i+1}s_is_{i+1}v>s_is_{i+1}v>s_{i+1}v>v$. Hence $s_is_{i+1}w=s_is_{i+1}s_is_{i+1}v=s_{i+1}s_iv$. But the former has length $l(v)+3$ and the latter $l(v)+2$, so $s_is_{i+1}w<s_{i+1}w$ and so $s_{i+1}w\notin D_L(s_i)$.
 
 %Otherwise, we have $s_{i+1}u>u$, and so $u=s_iw\notin D_L(s_i)$. We fix a reduced expression $\bar{u}$ for $u$. Then, since $s_{i+1}s_iu>s_iu>u$, we have that $s_{i+1}s_i\bar{u}$ is a reduced expression for $s_{i+1}s_iu$. Suppose for contradiction that $s_{i}s_{i+1}s_{i}u<s_{i+1}s_iu$. Then by \cref{existence enough for bruhat} we must have some decomposition $s_{i+1}s_i\bar{u}=\bar{v}s\bar{y}$ such that $\bar{v}\bar{y}$ is a reduced expression for $s_{i}s_{i+1}s_{i}u$. This implies that $vs=s_iv$. Hence, by inspection, $\bar{v}$ must be of the form $s_{i+1}s_i\bar{x}$. Thus $\bar{x}s\bar{y}=\bar{u}$, and also we have $s_{i+1}s_is_{i+1}x=s_{i+1}s_ixs$. Hence $s_{i+1}xy=xsy=u$, and so $xy=s_{i+1}u$. Thus $l(s_{i+1}u)=l(u)-1$, and so $s_{i+1}u<u$, a contradiction. Hence $s_{i}s_{i+1}s_{i}u>s_{i+1}s_iu=s_{i+1}w$ and so $s_{i+1}w\in D_L(s_i)$.
 
 %The claim for $D_R$ follows by inversion.
\end{proof}

\begin{definition}
        For $w\in D_L(s_i)$, let $\prescript{*}{}w\in \left\{s_iw, s_{i+1}w\right\}\cap D_{L}(s_i)$. The map $D_L(s_i)\rightarrow D_L(s_i)$ given by $w\mapsto \prescript{*}{}w$ is the left star operation on  $D_L(s_i)$.
        
        Similarly, for $w\in D_R(s_i)$, let $w^*\in \left\{ws_i, ws_{i+1}\right\}\cap D_{R}(s_i)$. The map $D_R(s_i)\rightarrow D_R(s_i)$ given by $w\mapsto w^*$ is the right star operation on  $D_R(s_i)$.
\end{definition}

The key ingredient that ensures the isomorphism we construct can be defined over $\zpmq$ is that the star operation changes the parity of the length. Indeed, we can be more precise.

\begin{prop}\label{stars are integral}
    $q^{\frac{{l(w)-l(w^*)+1}}{2}}$ and $q^{\frac{{l(w)-l(\prescript{*}{}w)+1}}{2}}$ are either $1$ or $q$. In particular, they are in $\zq$.
\end{prop}

\begin{proof}
This is equivalent to saying $l(w^*)=l(w)\pm1$ and $l(\prescript{*}{}w)=l(w)\pm1$, which both follow for $w\in\affw$ by \cite[Lemma 1.2]{lusztig2003hecke}. The claim for general $w$ then follows from \cref{weyl group affine decomposition} (3).
\end{proof}

The key property of the star operation is that it respects the cell structure. This is what enables us to use it to build isomorphisms between cells.

\begin{prop}\label{cells respect star operation}
Let $v,w\in D_L(s)$, and let $x,y\in D_R(s)$.
\begin{enumerate}
    \item $v\sim_Rw$ if and only if $\prescript{*}{}v\sim_R\prescript{*}{}w$.
    \item $x\sim_Ly$ if and only if $x^*\sim_Ly^*$.
    \item $w\sim_L \prescript{*}{}w$.
    \item $x\sim_R x^*$.
\end{enumerate}
\end{prop}

\begin{proof}
 This is \cite[Theorems 1.6.2-3]{shi1986kazhdan}.
\end{proof}

\begin{prop}\label{stars commute}
Let $w\in D_L(s)\cap D_R(t)$. Then $\prescript{*}{}w\in D_R(t)$, $w^*\in D_L(s)$, and $(\prescript{*}{}w)^*=\prescript{*}{}(w^*)$.
\end{prop}

\begin{proof}
 This is \cite[Lemma 1.4.3]{xi2002based}.
\end{proof}

We now construct the isomorphisms between the cells. To ensure they are coherent, we fix a particular left cell in each two-sided cell, and define an isomorphism between an arbitrary left cell and this fixed choice. Indeed, this choice has particular properties we will need in subsequent sections.

\begin{definition}
        Let $\lambda\vdash n$. Write $P_{\lambda}=\left\{s_{i}\in\fingen\middle|i\neq\sum_{j'=1}^j\lambda_j \ \text{for all} \ j\right\}$. Let $w_{\lambda}=w_{P_{\lambda}}$.
\end{definition}

\begin{prop}\label{w lambda in h lambda}
$\klu_{w_{\lambda}}\in\hecany^{\lambda}$.
\end{prop}

\begin{proof}
 This is \cite[Lemma 2.5]{shi1996partial}.
\end{proof}

\begin{definition}
        Let $\Gamma^{\lambda}$ denote the left cell of $\hecany$ containing $\klu_{w_{\lambda}}$. Similarly, let $\Psi^{\lambda}$ denote the left cell of $\schany$ containing $\klu_{\emptyset\emptyset}^{w_{\lambda}}$.
\end{definition}

Observe that, as $l(w_{\lambda}^{-1})=l(w_{\lambda})$ and $w_{\lambda}$ is unique of maximal length in $W_{P_{\lambda}}$, we have $w_{\lambda}^{-1}=w_{\lambda}$. Thus, $\iota\left(\Gamma^{\lambda}\right)$ is the right cell containing $\klu_{w_{\lambda}}$, and $\Gamma^{\lambda}, \iota\left(\Gamma^{\lambda}\right)\subseteq\hecany^{\lambda}$. Similarly, $\iota\left(\Psi^{\lambda}\right)$ is the right cell containing $\klu_{\emptyset\emptyset}^{w_{\lambda}}$, and $\Psi^{\lambda}, \iota\left(\Psi^{\lambda}\right)\subseteq\schany^{\lambda}$.

To define our isomorphism, we first recall a bijection between the Kazhdan-Lusztig bases of left cells. 

\begin{prop}\label{cell basis bijection}
Let $\Gamma\subseteq\hecany^{\lambda}$ be a left cell. Then there is an $i_{\Gamma}\in\z$ and a sequence of right star operations that gives a bijection \begin{align*}
    \left\{w\middle|\klu_w\in\Gamma\right\}\rightarrow\left\{w^{**}\middle|\klu_w\in\Gamma^{\lambda}\omega^{i_{\Gamma}}\right\}.
\end{align*}
\end{prop}

\begin{proof}
 This is \cite[Corollary 2.2.2]{xi2002based}.
\end{proof}

We are now ready to define our isomorphism.

\begin{definition}\label{maps between cells}
    Let $\Gamma$ be a left cell in $\hecany^{\lambda}$. We fix an $i_{\Gamma}$ and sequence of right star operations $w\mapsto w^{**}$ as in \cref{cell basis bijection}. In particular, for $\Gamma=\Gamma^{\lambda}$, we fix $i_{\Gamma}=0$ and the empty sequence $w\mapsto w$.
    
    Suppose the sequence has $j_{\Gamma}$ right star operations. Let $\phi_{\Gamma}$ denote the $\zany$-module homomorphism $\Gamma\rightarrow\Gamma^{\lambda}$ given by 
    \begin{align*}
        \klu_w\mapsto q^{\frac{{l(w)-l(w^{**})+j_{\Gamma}}}{2}}\klu_{w^{**}\omega^{-i_{\Gamma}}}.
    \end{align*}
    This is well-defined by \cref{stars are integral}.
    
    Now let $\Psi$ be a left cell in $\schany^{\lambda}$. By \cref{hecke and shcur cell relations}, this is given by a pair of a $P\subseteq\fingen$ and a left cell $\Gamma$ in $\hecany^{\lambda}$.
    Let $\phi_{\Psi}$ denote the $\zany$-module homomorphism $\Psi\rightarrow\Psi^{\lambda}$ given by 
    \begin{align*}
        \klu_{QP}^w\mapsto q^{\frac{{l(w)-l(w^{**})+j_{\Gamma}}}{2}}\klu_{Q\emptyset}^{w^{**}\omega^{-i_{\Gamma}}}.
    \end{align*}
    Again, this is well-defined by \cref{stars are integral}.
    
    Let $\phi_{\iota\left(\Gamma\right)}=\iota\phi_{\Gamma}\iota$. This is a $\zany$-module homomorphism $\iota\left(\Gamma\right)\rightarrow\iota\left(\Gamma^{\lambda}\right)$.

    Let $\phi_{\Gamma\cap\iota\left(\Delta\right)}=\phi_{\Gamma}\phi_{\iota\left(\Delta\right)}$. By \cref{cells respect star operation}, this is a $\zany$-module homomorphism $\Gamma\cap\iota\left(\Delta\right)\rightarrow\Gamma^{\lambda}\cap\iota\left(\Gamma^{\lambda}\right)$.
    
    Similarly, let $\phi_{\iota\left(\Psi\right)}=\iota\phi_{\Psi}\iota$, and let $\phi_{\Psi\cap\iota\left(\Phi\right)}=\phi_{\Psi}\phi_{\iota\left(\Phi\right)}$. Again by \cref{cells respect star operation}, these are $\zany$-module homomorphisms $\iota\left(\Psi\right)\rightarrow\iota\left(\Psi^{\lambda}\right)$ and $\Psi\cap\iota\left(\Phi\right)\rightarrow\Psi^{\lambda}\cap\iota\left(\Psi^{\lambda}\right)$ respectively.
\end{definition}

So far, we only know that this is a $\zany$-module homomorphism, not an isomorphism of left ideals.

\begin{lemma}\label{cells are module iso} We shall need the following properties:
\begin{enumerate}
    \item If $q^{-1}\in\zany$, then the maps in \cref{maps between cells} are isomorphisms.
    \item If $\Psi$ and $\Phi$ correspond to $(\Gamma,\emptyset)$ and $(\Delta,\emptyset)$ respectively via \cref{hecke and shcur cell relations}, then $\Psi\cap\iota\left(\Phi\right)=\Gamma\cap\iota\left(\Delta\right)$ and $\phi_{\Psi\cap\iota\left(\Phi\right)}=\phi_{\Gamma\cap\iota\left(\Delta\right)}$.
    \item $\iota\phi_{\Psi\cap\iota\left(\Phi\right)}=\phi_{\Phi\cap\iota\left(\Psi\right)}\iota$.
    \item If $q^{-1}\in\zany$, then $\phi_{\Psi_{\lambda}\cap\iota\left(\Phi\right)}^{-1}\phi_{\Psi\cap\iota\left(\Phi\right)}=\phi_{\Psi}$.
\end{enumerate}
\end{lemma}

\begin{proof}
The first claim follows from \cref{stars are integral} and \cref{cell basis bijection}. The second claim holds by definition of the various maps.

By \cref{stars commute}, we have that $\phi_{\Psi\cap\iota\left(\Phi\right)}=\phi_{\iota\left(\Phi\right)}\phi_{\Psi}$. Hence 
\begin{align*}
    \iota\phi_{\Psi\cap\iota\left(\Phi\right)} &= \iota \phi_{\Psi}\phi_{\iota\left(\Phi\right)}  \\
    &= \iota\phi_{\Psi}\iota\phi_{\Phi}\iota \\
    &=  \phi_{\iota\left(\Psi\right)}\phi_{\Phi}\iota \\
    &=\phi_{\Phi\cap\iota\left(\Psi\right)}\iota
\end{align*}
which gives the third claim, and
\begin{align*}
    \phi_{\Gamma_{\lambda}\cap\iota\left(\Phi\right)}^{-1}\phi_{\Psi\cap\iota\left(\Phi\right)} &= \left(\phi_{\Gamma_{\lambda}}\phi_{\Phi}^{-1}\right)^{-1}\phi_{\iota\left(\Phi\right)}\phi_{\Psi} \\
    &=\phi_{\Psi}
\end{align*}
gives the fourth.
\end{proof}

Thus our maps are isomorphisms. To see they are homomorphisms of left ideals, we use the following well-known result over $\zrtq$.

\begin{lemma}\label{normalised cell equivalence}
    Suppose $v,w\in D_R(s_i)$ satisfy $v\sim_Lw$. Then $\scn_{u,v}^w=\scn_{u,v^*}^{w^*}$.
\end{lemma}

\begin{proof}
 We follow the proof of \cite[Proposition 1.4.4(b)]{xi2002based}, which proves a closely related result. Let $s\in\left\{s_i,s_{i+1}\right\}$ satisfy $v^*s<v^*$, and let $t\in\left\{s_i,s_{i+1}\right\}\backslash\left\{s\right\}$. Then by \cite[Lemma 1.4.2(b)]{xi2002based} there exist elements $h_x\in\z$ such that
 \begin{align*}
     \kln_v\kln_s=\kln_{v^*}+\sum_{\mathclap{\substack{xs<x \\ xt<x}}}h_x\kln_x.
 \end{align*}
 Multiplying on the left by $\kln_u$ gives
 \begin{align*}
     \kln_u\kln_v\kln_s=\sum_{y\in\weyl}\left(\scn_{u,v^*}^y+\sum_{\mathclap{\substack{xs<x \\ xt<x}}}h_x\scn_{u,x}^y\right)\kln_y.
 \end{align*}
 However, we also have
 \begin{align*}
     \kln_u\kln_v\kln_s &=\sum_{z\in\weyl}\scn_{u,v}^z\kln_z\kln_s \\
     &=\sum_{y,z\in\weyl}\scn_{u,v}^z\scn_{z,s}^y\kln_y
 \end{align*}
 and so equating coefficients gives
 \begin{align*}
     \scn_{u,v^*}^y+\sum_{\mathclap{\substack{xs<x \\ xt<x}}}h_x\scn_{u,x}^y=\sum_{z\in\weyl}\scn_{u,v}^z\scn_{z,s}^y.
 \end{align*}
 
 Put $y=w^*$. Since $v\sim_Lw$ by assumption, we have by \cref{cells respect star operation} that $v^*\sim_L w^*$. Also by \cref{cells respect star operation} we have $v\sim_Rv^*$. Now, if $h_x\neq 0$ then $x\leq_R v$, which, together with the previous relation, gives $x\leq_{R}v^*$. Furthermore, if $h_{u,x}^{w^*}\neq 0$ then $w^*\leq_L x$, which together with the first relation gives $v^*\leq_L x$. Combining these gives that $v^*\sim_{LR}x$, and then by \cite[Corollary 1.9(c)]{lusztig1987cells} these last two relations gives $v^*\sim_L x$. Hence, by \cite[Proposition 2.4]{kazhdan1979representations}, $R(v^*)=R(x)$. But by assumption $s\notin R(v^*)$, and if $h_x\neq 0$ then $s\in R(x)$, a contradiction. Hence
 \begin{align*}
     \sum_{\mathclap{\substack{xs<x \\ xt<x}}}h_x\scn_{u,x}^{w^*}=0.
 \end{align*}
 
 Now fix $z\in\weyl$. By \cref{cells respect star operation} again we have $w\sim_Rw^*$. Hence if $\scn_{u,v}^z\neq 0$ then $z\leq_L v$ and so $z\leq_L w$, and if $\scn_{z,s}^{w*}\neq 0$ then $w*\leq_R z$ and so $w\leq_R z$. The same logic as the previous case then gives that $z\sim_L w\sim_L v$, and so $R(z)=R(w)=R(v)$. Thus $t\in R(z)$ and $s\notin R(z)$, so $z^*$ exists and furthermore $z^*s<z^*$. Thus by \cite[Lemma 1.4.2(b)]{xi2002based} again there exist elements $h_x'\in\z$ such that
 \begin{align*}
     \kln_z\kln_s=\kln_{z^*}+\sum_{\mathclap{\substack{xs<x \\ xt<x}}}h_x'\kln_x.
 \end{align*}
 
In particular, as $s\in R(w^*)$, we must have $h_{w^*}'=0$. Thus if $\kln_{z,s}^{w^*}\neq 0$ then we must have $w^*=z^*$, and furthermore in this case $\klu_{z,s}^{w*}=1$. But if $w^*=z^*$ then $w=z$. Hence
\begin{align*}
    \sum_{z\in\weyl}\scn_{u,v}^z\scn_{z,s}^y=\scn_{u,v}^w.
\end{align*}
\end{proof}

We now translate this result to $\zany$, which gives us what we need.

\begin{thm}\label{left cells are iso}
Let $\Psi$ and $\Gamma$ be left cells in $\schany^{\lambda}$ and $\hecany^{\lambda}$ respectively. Then $\phi_{\Psi}$ and $\phi_{\Gamma}$) are homomorphisms of left $\schany/\schany^{>\lambda}$-modules and $\hecany/\hecany^{>\lambda}$-modules respectively.
\end{thm}

\begin{proof}
By construction of $\phi_{\Psi}$ and $\phi_{\Gamma}$, and by \cref{schur algebra structure coefficients}, it suffices to show for all $v\sim_L w$ that if $v,w\in D_R(s)$ then $q^{\frac{l(w)-l(w*)+1}{2}}\scu_{u,v}^w=q^{\frac{l(v)-l(v*)+1}{2}}\scu_{u,v*}^{w*}$, and also that if $i\in z$ then $\scu_{u,v}^w=\scu_{u,v\omega^i}^{w\omega^i}$. The latter is immediate from \cref{structure constant facts}, so we focus on the former, which is a simple computation using \cref{normalised cell equivalence}:
\begin{align*}
    q^{\frac{l(w)-l(w*)+1}{2}}\scu_{u,v}^w &= q^{\frac{l(u)+l(v)-l(w*)+1}{2}}\scn_{u,v}^w\\
    &= q^{\frac{l(u)+l(v)-l(w*)+1}{2}}\scn_{u,v*}^{w*} \\
    &= q^{\frac{l(v)-l(v*)+1}{2}}\scu_{u,v*}^{w*}.
\end{align*}
\end{proof}

\section{A renormalised integral lattice in the asymptotic algebra}

The piece of the affine cellular structure is a commutative $\zpmq$-algebra. Over $\zrtq$, this is provided by the asymptotic algebra. In this section, we provide a suitable candidate over $\zpmq$. We start by recalling the construction of the asymptotic algebra. For this, we must in turn first recall the construction of Lusztig's $a$-function.

\begin{definition}
    Let $n_{u,v}^w$ be the greatest power of $q^{\frac{1}{2}}$ occurring in $\scn_{u,v}^w$. Similarly, let $n_{uRv}^{QwP}$ be the greatest power of $q^{\frac{1}{2}}$ occurring in $\scn_{uRv}^{QwP}$.

    Let $a(w)=\max\left\{n_{u,v}^w\middle|u,v\in\weyl\right\}$, or infinity if no such maximum exists.
\end{definition}

We recall the properties of $a$ that we need.

\begin{prop}\label{a is finite}This is integer-valued:
    \begin{enumerate}
        \item $a(w)$ is finite for all $w\in\weyl$, and moreover is bounded by $l(w_{\fingen})$.
        \item \begin{align*}
            \max\left\{n_{uRv}^{QwP}+l(w_R)-l(w_P)\middle|R\subseteq\fingen,u\in\prescript{Q}{}\weyl^R,v\in\prescript{R}{}\weyl^P\right\} \\
            = a(w)-l(w_P).
        \end{align*}
        \item If $v\leq_{LR}w$ then $a(v)\geq a(w)$. In particular, $a(\omega^iw\omega^j)=a(w)$.
    \end{enumerate}
\end{prop}

\begin{proof}
 The first claim is \cite[Corollary 7.3]{lusztig1985cells} for $w\in\affw$. That $a(w)=a(w\omega^i)$ follows from \cref{structure constant facts}, and from this the first claim follows in general. The second claim is \cite[Proposition 3.8(d)]{mcgerty2003cells}, noting that Definition 3.1 therein has a missing factor of $\frac{1}{2}$, and the correct formula is the one we give here. The third claim is \cite[Theorem 5.4]{lusztig1985cells} in the affine case, and the general case then follows from $a(w)=a(w\omega^i)$ and \cref{cells on the affine subgroup}. %The fourth claim is \cite[Theorem 2.2]{lusztig1995quantum} in the affine case, and the general case again follows by the first claim and \cref{structure constant facts}. The final claim is \cite[Lemma 4.12]{mcgerty2003cells}, again noting that Definition 4.11 therein should only quantify over $\schany^{\lambda}$ and not all of $\schany$.
\end{proof}

We also introduce the objects $\gamma$, which will turn out to be the structure coefficients of the asymptotic algebra, and collect some of their basic properties.

\begin{definition}
    Write $\gamma_{u,v}^w$ for the coefficient of the $a(w)$-power term in $\scn_{u,v}^w$. Similarly, write $\gamma_{uRv}^{QwP}$ for the coefficient of the $a(w)-l(w_R)$-power term in $\scn_{uRv}^{QwP}$.
\end{definition}

\begin{lemma}\label{asymptotic structure constants}
    $\gamma_{uRv}^{QwP}=\gamma_{u,v}^w$.
\end{lemma}

\begin{proof}
 This follows from \cref{schur algebra structure coefficients} together with \cref{a is finite} (2).
\end{proof}

\begin{lemma}\label{affine asymptotic structure}
Let $u',v'\in\affw$. Then $\gamma_{u'\omega^i ,v'\omega^j}^w$ is zero unless $w=w'\omega^{i+j}$ for some $w'\in\affw$, in which case $\gamma_{u'\omega^i ,v'\omega^j}^w=\gamma_{u',\omega^i v'\omega^{-i} }^{w'}$. Furthermore, $\gamma_{\omega^{i}u'\omega^{-i} ,\omega^{i}v'\omega^{-i}}^{\omega^{i}w'\omega^{-i}}=\gamma_{u', v'}^{w'}$.
\end{lemma}

\begin{proof}
 The first part is \cref{structure constant facts} together with \cref{a is finite} (3). The second part is \cref{structure constant facts} again together with \cref{a is finite} (3).
\end{proof}

Using $a$, we now define the asymptotic algebra of a two-sided cell as a certain $\z$-subquotient of the cell.

\begin{definition}
     For $w\in\weyl$, write $\kla_{w}=q^{-\frac{a(w)}{2}}\kln_w$. Similarly, for $P,Q\subseteq\fingen$ and $w\in\prescript{Q}{}\weyl^P$, write $\kla_{QP}^{w}=q^{\frac{l(w_P)-a(w)}{2}}\kln_{QP}^w$.

Write $\zmq=\z\left[q^{-\frac{1}{2}}\right]$, the unital subalgebra of $\zrtq$ generated by $q^{-\frac{1}{2}}$.
     
     Let $\hecmq^{\lambda}$ and $\schmq^{\lambda}$ denote the $\zmq$-span of the $\kla_w$ and the $\kla_{QP}^w$ respectively.
\end{definition}

\begin{lemma}
    $\hecmq^{\lambda}$ and $\schmq^{\lambda}$ are free $\zmq$-subalgebras of $\hecrtq^{\lambda}$ and $\schrtq^{\lambda}$ respectively.
\end{lemma}

\begin{proof}
 Let $a=a(w)$ for any $\klu_w\in\hecany^{\lambda}$, noting that by \cref{a is finite} (3) it is independent of choice of $w$. Then, for $\klu_u,\klu_v,\klu_w\in\hecany^{\lambda}$, the coefficient of $\kla_w$ in the product $\kla_u\kla_v$ is $q^{-\frac{a}{2}}\scn_{u,v}^w$, which by definition of $a$ must lie in $\zmq$. Similarly, the coefficient of $\kla_{QP}^w$ in the product $\kla_{QR}^u\kla_{RP}^v$ is $q^{l(w_R)-\frac{a}{2}}\scn_{uRv}^{QwP}$, which by \cref{a is finite} (2) must lie in $\zmq$.
\end{proof}

\begin{definition}
     Let $\hecinf^{\lambda}=\hecmq^{\lambda}/q^{-\frac{1}{2}}\hecmq^{\lambda}$ and $\schinf^{\lambda}=
     \schmq^{\lambda}/q^{-\frac{1}{2}}\schmq^{\lambda}$, the asymptotic algebras of $\hecany$ and $\schany$ respectively.
\end{definition}

\begin{lemma}
      $\hecinf^{\lambda}$ and $\schinf^{\lambda}$) are free $\z$-algebras with bases $\left\{t_w\middle|\klu_w\in\hecany^{\lambda}\right\}$ and $\left\{t_{QP}^w\middle|\klu_{QP}^w\in\schany^{\lambda}\right\}$ respectively, where $t_w$ is the image of $\kla_w$ and $t_{QP}^w$ is the image of $\kla_{QP}^w$. The products are $t_ut_v=\sum_{w\in\weyl}\gamma_{u,v}^w t_w$ and $t_{QR}^ut_{RP}^v=\sum_{w\in\weyl}\gamma_{uRv}^{QwP} t_{QP}^w$.
      
      The map $t_w\mapsto t_{\emptyset\emptyset}^w$ is an inclusion of $\z$-algebras $\hecinf\rightarrow\schinf$.
\end{lemma}

\begin{proof}
 That they are free $\z$-algebras with the given basis is immediate by construction. Write $a=a(w)$ for any $\klu_w\in\hecany^{\lambda}$. Then by \cref{a is finite} (2) we have
 \begin{align*}
     q^{-\frac{a}{2}}\scn_{u,v}^w\cong \gamma_{u,v}^w  \ \left(\text{mod} \ q^{-\frac{1}{2}}\zmq\right) \\ q^{l(w_R)-\frac{a}{2}}\scn_{uRv}^{QwP} \cong \gamma_{uRv}^{QwP}  \ \left(\text{mod} \ q^{-\frac{1}{2}}\zmq\right)
 \end{align*}
 which gives the stated product formulae. That the map $t_w\mapsto t_{\emptyset\emptyset}^w$ is an inclusion follows from \cref{asymptotic structure constants}.
\end{proof}

We will need the following key property.

\begin{prop}\label{asymptotic algebra sum of cells}
    If $\gamma_{u,v}^w\neq 0$ then $v\sim_{L}w\sim_{R}u\sim_{L}v^{-1}$. 
\end{prop}

\begin{proof}
\cite[Theorem 1.9(a)]{lusztig1987cells} gives the case where $u,v,w\in\affw$. In the general case, if $\gamma_{u'\omega^i ,v'\omega^j}^w\neq0$ then by \cref{affine asymptotic structure} we have $\gamma_{u',\omega^i v'\omega^{-i} }^{w'}\neq0$ for $w=w'\omega^{i+j} $, and so by the affine case and \cref{cells on the affine subgroup} we have $v'\omega^j\sim_L w'\omega^{i+j} \sim_R u'\omega^i \sim_L \omega^{-j}v'^{-1}$.
\end{proof}

We also record the following easy consequence.

\begin{cor}
$\hecinf=\bigoplus_{\lambda\vdash n}\hecinf^{\lambda}$ is the free $\z$-algebra with basis $\left\{t_w\middle|w\in\weyl\right\}$ and product $t_ut_v=\sum_{w\in\weyl}\gamma_{u,v}^w t_w$, and $\schinf=\bigoplus_{\lambda\vdash n}\schinf^{\lambda}$ is the free $\z$-algebra with basis $\left\{t_{QP}^w\middle|P,Q\subseteq\fingen,w\in\prescript{Q}{}\weyl^P\right\}$ and product $t_{QR}^ut_{RP}^v=\sum_{w\in\weyl}\gamma_{uRv}^{QwP} t_{QP}^w$.
\end{cor}

\begin{proof}
 By \cref{asymptotic algebra sum of cells}, if $\klu_u,\klu_v,\klu_w$ are not all in the same two-sided cell, then $\gamma_{u,v}^w=0$. Together with \cref{asymptotic structure constants}, this says that $t_{QR}^ut_{RP}^v=0$ unless $\klu_{QR}^u,\klu_{RP}^v\in\schany^{\lambda}$ for some $\lambda$, in which case the terms in the sum $\sum_{w\in\weyl}\gamma_{uRv}^{QwP} t_{QP}^w$ are zero unless $\klu_{QP}^w\cong\schany^{\lambda}$ also.
\end{proof}

The asymptotic algebra also possesses an involution, which will be a part of the affine cellular structure.

\begin{definition}
     Write $\hecia=\hecinf\otimes_{\z}\zany$ and $\schia=\schinf\otimes_{\z}\zany$.
     
     Let $\sigma(t_{QP}^w)=t_{PQ}^{w^{-1}}$.
     
\end{definition}

\begin{lemma}
$\sigma$ is an an algebra anti-involution on $\schia$, which restricts to an algebra anti-involution of $\hecia$.
\end{lemma}

\begin{proof}
 By \cref{cellular involution}, we have that $\scu_{uRv}^{QwP}=\scu_{v^{-1}Ru^{-1}}^{Qw^{-1}P}$. Furthermore, by \cref{cells respect involution}, we have $w\sim_{LR}w^{-1}$, and so by \cref{a is finite} (3) we have $a(w)=a(w^{-1})$. Hence $\ascu_{uRv}^{QwP}=\ascu_{v^{-1}Ru^{-1}}^{Qw^{-1}P}$.
\end{proof}

We can also define asymptotic versions of left and right cells. In particular, we need the intersection of our fixed choice of left and right cells.

\begin{definition}

     If $\Gamma$ and $\Psi$ are left cells of $\hecany$ and $\schany$ respectively, write $\Gamma_{\infty}$ and $\Psi_{\infty}$ for the $\zany$-span of $\left\{t_{w}\middle|\klu_w\in\Gamma\right\}$ and $\left\{t_{QP}^{w}\middle|\klu_{QP}^w\in\Psi\right\}$ respectively.
     
     Write $\tcany=\Gamma_{\infty}^{\lambda}\cap\sigma\left(\Gamma_{\infty}^{\lambda}\right)=\Psi_{\infty}^{\lambda}\cap\sigma\left(\Psi_{\infty}^{\lambda}\right)$.
\end{definition}

\begin{lemma}
    For any $\Gamma$ and $\Psi$, the spaces $\Gamma_{\infty}$ and $\Psi_{\infty}$ are left ideals in $\hecia$ and $\schia$ respectively. $\tcany$ is a subalgebra of $\hecia$. Furthermore, $\sigma$ restricts to an algebra anti-involution on $\tcany$.
\end{lemma}

\begin{proof}
 If $\ascu_{u,v}^w\neq 0$ then by \cref{asymptotic algebra sum of cells} we must have $v\sim_L w$. This says that $\Gamma_{\infty}$ is a left ideal. Furthermore, combining this with \cref{asymptotic structure constants} gives that if $\ascu_{uRv}^{QwP}\neq 0$ then $v\sim_L w$. Thus by \cref{hecke and shcur cell relations} we get that $\Psi_{\infty}$ is a left ideal. The second claim then follows as $\tcany$ is the intersection of a right and a left ideal. The final claim is immediate by definition of $\tcany$ and $\sigma$.
\end{proof}

We are now ready to recall the explicit description of $\hecia$ and $\schia$ as sums of matrix algebras over polynomial rings.

\begin{definition}
     For $m\in\z_{\geq0}$, let
     \begin{align*}
         \z_{\domin}^m=\left\{(a_1,\dots,a_m)\in\z^m\middle|a_j\geq a_{j+1},\ 1\leq j\leq m\right\}.
     \end{align*}
     
     For $\lambda\vdash n$, fix $r_1<\dots r_p$ such that, for all $r$, we have that $\lambda_{r}=\lambda_{r_i}$ for exactly one $r_i$, and such that $r_i$ is maximal among the $r$ with $\lambda_r=\lambda_{r_i}$. Set $m_i=\lambda_{r_i}-\lambda_{r_{i+1}}$, taking $r_{p+1}=0$, and write $\domin\left(\lambda\right)=\prod_{i}\z_{\domin}^{m_i}$.
     
    Write $y_{ij}$ for the element of $\domin(\lambda)$ with $(i,j')$-component $1$ if $j'\leq j$ and all other components $0$.
\end{definition}

\begin{remark}
      Thus $\domin\left(\lambda\right)$ is generated under addition by
      
     \begin{align*}
         \left\{y_{ij}\middle|1\leq i\leq p, 1\leq j\leq m_i\right\}\cup\left\{-y_{im_i}\middle|1\leq i\leq p\right\}. 
     \end{align*}%Thus, a general $x\in\domin\left(\lambda\right)$ can be written uniquely as $x=\prod_{i,j}e_{ij}^{a_{ij}}$ for $a_{ij}\in\z_{\geq 0}$ when $j < m_i$ and $a_{im_i}\in\z$, where we abbreviate $\prod_{i,j}=\prod_{i=1}^p\prod_{j=1}^{m_i}$.
\end{remark}

\begin{prop}\label{normalised asymptotic cell}
    There is a bijection $\domin\left(\lambda\right)\rightarrow \left\{w\in\weyl\middle| t_{w}\in \tcany\right\}$, $x\mapsto w(x)$, such that the product on $\tcany$ is given by $t_{w(x)}t_{w(y_{ij})}=\sum_{\tau+x\in\domin(\lambda)}t_{\tau+x}$, where $\tau$ runs over all tuples (not necessarily in $\domin(\lambda)$) that have exactly $j$-many $(i,j')$-entries which are $1$ and all other entries $0$.
    
    %Furthermore, if $x=(x_i)_{i=1}^p$, then $t_{w(x)}=\prod_{i=1}^{p}t_{w(x_i)}$
    Furthermore, $\tcany=\zany\left[t_{w(y_{ij})},t_{w(y_{im_i})}^{-1}\middle|1\leq i\leq p,1\leq j\leq m_i\right]$.
    
    The algebra $\schia^{\lambda}$ is isomorphic to the $m_{\lambda}$-by-$m_{\lambda}$ matrix algebra over $\tcany$, with rows and columns indexed by the left cells in $\schany^{\lambda}$, via the map sending $\asn_{QP}^w\in {\Psi\cap\iota\left(\Phi\right)}$ to the matrix which is nonzero only in row $\Phi$ and column $\Psi$, and whose value in that column is $\asn_{w'}$, where $\klu_{w'}=q^{m}\psi_{\Psi\cap\iota\left(\Phi\right)}\left(\klu_{QP}^{w}\right)$ for some $m\in\frac{1}{2}\z$.
\end{prop}

\begin{proof}
The first part is \cite[Theorem 5.2.6(b), Theorem 6.4.1, Proposition 8.1.3]{xi2002based}. Furthermore, by \cite[Section 4.2(d,f)]{xi2002based}, we have that $\tcany$ isomorphic to the extension of scalars to $\zany$ of the representation ring of $\prod_{i=1}^{p}\mathrm{GL}_{m_i}(\mathbb{C})$, which is in turn isomorphic to the claimed algebra via standard results: see for example \cite[Exercise 23.36(d)]{fulton1991representation}. The final claim is \cite[Proposition 4.13]{mcgerty2003cells}.
\end{proof}

The usual affine cellular structure over $\zrtq$ is built from a correspondence between the $\asn_w$ and the $\kln_w$. Hence, to get a structure over $\zpmq$, we seek a basis of $\hecirt$ and $\schirt$ analogous to $\klu_w$, as this is will turn out to play more nicely with $\hecpmq$ and $\schpmq$. The challenge is ensuring that this basis spans a $\zpmq$-subalgebra of $\tcrtq$. To do this, we need to understand the lengths of the $w(x)$ in terms of $x$.

For $k\geq 1$ and $1\leq l\leq\lambda_k$, write $e_{k,l}=l+\sum_{k'=1}^{k-1}\lambda_{k'}$.

\begin{prop}
    $w_{\lambda}(e_{k,l})=e_{k,\lambda_k-l+1}$ for all $1\leq k\leq r_p$ and $1\leq l\leq\lambda_k$. Furthermore, $w_{\lambda}=w(0)$.
\end{prop}

\begin{proof}
By definition, $w_{\lambda}\leq w_{\lambda}s_k$ for exactly those $k\neq\lambda_i$ for some $i$. Hence by \cite[Lemma 2.1.3(f)]{xi2002based}, we have $w(k)>w(k+1)$ if and only if $k\neq\lambda_i$ for some $i$. Furthermore, as $w_{\lambda}$ is a product pf $s_k$, it must permute the elements $\left\{1,\dots,n\right\}$. Thus it is given by the above formula. (Alternatively, see \cite[Section 5.5]{xi2002based}).

By \cite[Section 5.3.5]{xi2002based}, $w_{\lambda}=w(0)$.
\end{proof}

\begin{lemma}\label{lengths of the lattice}
Suppose $x=(x_{i',j'})\in\domin(\lambda)$ has $x_{i'j'}\geq0$ for all $i',j'$. Then

\begin{enumerate}
    \item $w=w(x)$ satisfies $w(e_{k,l})>w(e_{k,l+1})$ for all $1\leq l\leq \lambda_k-1$.
    \item $l(w(x))\in l(w_{\lambda})+\sum_{i',j'}x_{i'j'}(n+r_{i'})+2\z$.
\end{enumerate}

\end{lemma}

\begin{proof}
Both properties manifestly hold for $w(0)=w_{\lambda}$. Hence, by induction, it suffices to assume that, if $i$ is maximal such that $x_{ij'}>0$ for some $j'$, and if $j$ is maximal such that $x_{ij}>0$, we have already proven both claims for $w'=w(x')$, where $x'_{ij}=x_{ij}-1$ and $x'_{i'j'}=x_{i'j'}$ for all other $i',j'$. But by \cite[Lemma 5.3.1]{xi2002based}, we have that 
\begin{align}\label{w' from w}
    w(a)=\begin{cases}
    w'(e_{k+1,j_{k+1}}), &a=e_{k,j_k}, \ 1\leq k <r_i \\
    w'(e_{1,j_1})+n, &a=e_{r_i,j_i} \\
    w'(a), &\text{otherwise}
    \end{cases}
\end{align}
where $j_k$ satisfy $j_{r_i}=j$ and
\begin{align}\label{bijection recurrence}
    w'(e_{k,j_k-1})>w'(e_{k+1,j_{k+1}})>w'(e_{k,j_k}) \ \text{for} \ 1\leq k\leq r_i-1,
\end{align}

and furthermore, $w(a)=w'(a)=w_{\lambda}(a)$ if $a>e_{r_i,j}$, and $w(a),w'(a)>0$ if $a>0$.

We consider the first claim. As $w'(e_{k,l})>w'(e_{k,l+1})$, and $w(a)=w'(a)$ for $a\neq e_{k,j_k}$, $1\leq k\leq r_i$, it only remains to check that \begin{align*}
    w'(e_{k,j_k-1})>w'(e_{k+1,j_{k+1}})>w'(e_{k,j_k+1}). 
\end{align*}
But the first of these is exactly the first inequality in \cref{bijection recurrence} above, and by the second inequality in \cref{bijection recurrence} we have $w'(e_{k+1,j_{k+1}})>w'(e_{k,j_k})>w'(e_{k,j_k+1})$. Thus we have the first claim.

To see the second claim, it suffices to show that $l(w)\in l(w')+n+r_i+2\z$. Recall that $l(w')=\sum_{1\leq a<b\leq n}\left|\left\lfloor\frac{w'(b)-w'(a)}{n}\right\rfloor\right|$. Thus we need to determine how the terms in this sum change value upon the substitution $w'\mapsto w$. Let $\sigma$ denote the permutation of $1,\dots,n$ given by $\sigma(e_{k+1,j_{k+1}})=e_{k,j_k}$ for $1\leq k\leq r_i-1$ and $\sigma(e_{1,j_1})=e_{r_i,j_{r_i}}$, and $\sigma(a)=a$ otherwise. We shall compare the $(a,b)$-term in the sum for $l(w')$ to the $(\min(\sigma(a),\sigma(b)),\max(\sigma(a),\sigma(b)))$-term in the sum for $l(w)$. Observe that $w'(a)=w(\sigma(a))$ unless $a=e_{1,j_1}$, in which case $w'(a)+n=w(\sigma(a))$. Hence if $a,b\neq e_{1,j_1}$ then the only way the terms can be different is if $\sigma(b)<\sigma(a)$.

For $a$ and $b$ both not of the form $e_{k,j_k}$ for some $1\leq k\leq r_i$, we have $\sigma(a)<\sigma(b)$, and hence the terms are equal.

Let $1\leq k\leq r_{i-1}$. If $a=e_{k+1,j_{k+1}}$ and $b$ is arbitrary, we have $\sigma(a)<\sigma(b)$ and hence the terms are equal. Similarly, if $b=e_{k+1,j_{k+1}}$ and $a\neq e_{1,j_1}$ we have $\sigma(a)<\sigma(b)$ unless $a>e_{k,j_k}$, in which case $\sigma(a)>\sigma(b)$. If $e_{k,j_k}<a\leq e_{k,\lambda_k}$, then by the first claim we have $w'(a)=w(a)<w(e_{k,j_k})=w'(e_{k+1,j_k+1})$, and so the term changes from $\left\lfloor\frac{w'(e_{k+1,j_k+1})-w'(a)}{n}\right\rfloor$ to \begin{align*}
    \left|\left\lfloor\frac{w'(a)-w'(e_{k+1,j_k+1})}{n}\right\rfloor\right|&=\left\lceil\frac{w'(e_{k+1,j_k+1})-w'(a)}{n}\right\rceil\\
    &=\left\lfloor\frac{w'(e_{k+1,j_k+1})-w'(a)}{n}\right\rfloor+1.
\end{align*}
Meanwhile, if $a\geq e_{k+1,1}$, then by the first claim again we have $w(a)=w'(a)>w'(e_{k+1,j_k+1})=w(e_{k,j_{k}})$, and so the term changes from
\begin{align*}
    \left|\left\lfloor\frac{w'(e_{k+1,j_k+1})-w'(a)}{n}\right\rfloor\right|&=\left\lceil\frac{w'(a)-w'(e_{k+1,j_k+1})}{n}\right\rceil\\
    &=\left\lfloor\frac{w'(a)-w'(e_{k+1,j_k+1})}{n}\right\rfloor-1
\end{align*}
to $\left\lfloor\frac{w'(a)-w'(e_{k+1,j_k+1})}{n}\right\rfloor$.

If $a=e_{1,j_1}$ and $b>e_{r_i,j}$, then $\sigma(a)<\sigma(b)$ and $w(b)=w'(b)=w_{\lambda}(b)\leq n$, and so as $w'(e_{1,j_1})>0$ the term changes from
\begin{align*}
    \left|\left\lfloor\frac{w'(b)-w'(e_{1,j_1})}{n}\right\rfloor\right|=\left\lceil\frac{w'(e_{1,j_1})-w'(b)}{n}\right\rceil
\end{align*}
to
\begin{align*}
    \left|\left\lfloor\frac{w'(b)-w'(e_{1,j_1})-n}{n}\right\rfloor\right|=\left\lceil\frac{w'(e_{1,j_1})-w'(b)}{n}\right\rceil+1.
\end{align*}

If $a=e_{1,j_1}$ and $e_{r_1,j_1}<b\leq e_{r_i,j_{r_i}}$, then $\sigma(a)>\sigma(b)$, and so the term changes from $\left|\left\lfloor\frac{w'(b)-w'(e_{1,j_1})}{n}\right\rfloor\right|$ to $\left|\left\lfloor\frac{w'(e_{1,j_1})+n-w'(b)}{n}\right\rfloor\right|$. If $w'(b)>w'(e_{1,j_1})$ then \begin{align*}
    \left|\left\lfloor\frac{w'(b)-w'(e_{1,j_1})}{n}\right\rfloor\right|=\left\lfloor\frac{w'(b)-w'(e_{1,j_1})}{n}\right\rfloor
\end{align*}
and
\begin{align*}
    \left|\left\lfloor\frac{w'(e_{1,j_1})+n-w'(b)}{n}\right\rfloor\right|=\left\lceil\frac{w'(b)-w'(e_{1,j_1})}{n}\right\rceil-1=\left\lfloor\frac{w'(b)-w'(e_{1,j_1})}{n}\right\rfloor
\end{align*}
while if $w'(b)<w'(e_{1,j_1})$ then
\begin{align*}
    \left|\left\lfloor\frac{w'(b)-w'(e_{1,j_1})}{n}\right\rfloor\right|=\left\lceil\frac{w'(e_{1,j_1})-w'(b)}{n}\right\rceil
\end{align*}
 and
 \begin{align*}
     \left|\left\lfloor\frac{w'(e_{1,j_1})+n-w'(b)}{n}\right\rfloor\right|=\left\lfloor\frac{w'(e_{1,j_1})-w'(b)}{n}\right\rfloor+1=\left\lceil\frac{w'(e_{1,j_1})-w'(b)}{n}\right\rceil.
 \end{align*}
Thus in both cases the term does not change.

Finally. if $b=e_{r_1,j_1}$ then by the first claim $w'(a)>w'(e_{1,j_1})$, and so the contribution changes from
\begin{align*}
    \left|\left\lfloor\frac{w'(e_{1,j_1})-w'(a)}{n}\right\rfloor\right|=\left\lceil\frac{w'(a)-w'(e_{1_{j_1}})}{n}\right\rceil
\end{align*}
to
\begin{align*}
    \left|\left\lfloor\frac{w'(e_{1,j_1}+n-w'(a)}{n}\right\rfloor\right|=\left\lceil\frac{w'(a)-w'(e_{1_{j_1}})}{n}\right\rceil-1.
\end{align*}

        Putting this all together, the change in length is
        \begin{align*}
          l(w)-l(w') = &\sum_{k=1}^{r_i-1}\left((\lambda_k-j_k)-(j_{k+1}-1)\right)+\left(n-j_{r_i}-\sum_{k=1}^{r_i-1}\lambda_k\right)-(j_1-1) \\
            = &n+r_i-2\sum_{k=1}^{r_i}j_k \\
            \in &n+r_i+2\z.
        \end{align*}
\end{proof}

We now have everything we need to define our new basis and show that it is a subalgebra.

\begin{definition}
     For $w\in\weyl$ such that $\asn_w\in \tcany$, define a new element 
     \begin{align*}
         \asu_w=q^{\frac{l(w)-l(w_{\lambda})}{2}}\asn_w\in\tcrtq.
     \end{align*}
     
     Let $\tcuany$ denote the $\zany$-span of $\left\{\asu_{w}\middle|\asn_w\in\tcany\right\}$. Thus
     \begin{align*}
         \tcurtq=\tcrtq=\zrtq\left[t_{w(y_{ij})},t_{w(y_{im_i})}^{-1}\middle|1\leq i\leq p,1\leq j\leq m_i\right].
     \end{align*}
\end{definition}

\begin{thm}\label{properties of B}
$\tcupmq=\zpmq\left[\asu_{w(y_{ij})},\asu_{w(y_{im_i})}^{-1}\middle|1\leq i\leq p,1\leq j\leq m_i\right]$. In particular, it is a commutative and finitely generated $\zpmq$-subalgebra of $\tcurtq$.
\end{thm}

\begin{proof}
To show $\tcupmq$ is a subalgebra, it suffices to show that the product $\asu_{w(x)}\asu_{w(y_{ij})}$ lies in $\tcupmq$. But
\begin{align}\label{integral lattice equation}
\asu_{w(x)}\asu_{w(x_{ij})}=q^{\frac{l(w(x))+l(w(y_{ij}))-l(w_{\lambda)}}{2}}\sum_{\tau+x\in\domin(\lambda)}q^{\frac{-l(w(\tau+x))}{2}}\asu_{w(\tau+x)}
\end{align}
where $\tau$ runs over all tuples that have exactly $j$-many $(i,j')$-entries which are $1$ and all other entries $0$. Hence the goal is to show that \begin{align*}
    l(w(x))+l(w(y_{ij}))-l(w_{\lambda})-l(w(\tau+x))\in2\z.
\end{align*}

Observe that if $\asn_{w}\in\tcany$ then by \cref{omega to the n} we have $\asn_{\omega^nw}\in\tcany$. Indeed, by the construction of $x\mapsto w(x)$ (see \cite[Section 5.2.1]{xi2002based}), $\omega^n w(x)=w(x')$, where $x'_{ij}=x_{ij}+r_i$. Furthermore, this also gives that $\omega^n w(\tau+x)=w(\tau+x')\in\tcany$, and so since $l(w(x))=l(w(x'))$ and $l(w(\tau+x))=l(w(\tau+x'))$, we may replace $x$ with $x'$ in the above formula. Iterating this sufficiently many times, we may assume that $x_{ij}\geq 0$ for all $i,j$.

But then by \cref{lengths of the lattice}, we have:
\begin{align*}
    l(w(x)) &\in l(w_{\lambda})+\sum_{i',j'}x_{i'j'}(n+r_{i'})+2\z \\
    l(w(y_{ij})) &\in l(w_{\lambda})+j(n+r_{i}) +2\z \\
    l(w(\tau+x)) &\in l(w_{\lambda})+\sum_{i',j'}x_{i'j'}(n+r_{i'})+j(n+r_i) + 2\z
\end{align*}
so we indeed have $l(w(x))+l(w(y_{ij}))-l(w_{\lambda})-l(w(\tau+x))\in2\z$.

Now, $\tcurtq=\tcrtq=\zrtq\left[t_{w(y_{ij})},t_{w(y_{im_i})}^{-1}\middle|1\leq i\leq p,1\leq j\leq m_i\right]$. Hence, by rescaling $\asn_{w(y_ij)}\mapsto\asu_{w(y_ij)}$ we have $\tcurtq=\zrtq\left[\asu_{w(y_{ij})},\asu_{w(y_{im_i})}^{-1}\right]$. As $\tcupmq$ is a subalgebra we have $\zpmq\left[\asu_{w(y_{ij})},\asu_{w(y_{im_i})}^{-1}\right]\subseteq\tcupmq$. To see the reverse inclusion, observe that if $x\in\tcupmq$, we can write it as a $\zrtq$-linear combination of products of the $\asu_{w(y_{ij})}$ and $\asu_{w(y_{im_i})}^{-1}$. But a product of the $\asu_{w(y_{ij})}$ and $\asu_{w(y_{im_i})}^{-1}$ is always in $\tcupmq$, and hence a $\zrtq$-linear combination of such products will lie in $\tcupmq$ precisely when the terms with coefficients an odd power of $q^{\frac{1}{2}}$ all sum to zero. Hence, the sum of only the terms with coefficients in $\tcupmq$ will sum to $x$.
\end{proof}

\begin{remark}
We have that $\asn_{w(y_{11})}^2=\asn_{w(2y_{11})}+\asn_{w(y_{12})}$, and by \cite[Section 5.5]{xi2002based} we have $l(w(y_{11}))=l(w_{\lambda})+n-r_1$ and $l(w(y_{12}))=l(w_{\lambda})+2n-4r_1$, so the coefficient of $\asu_{w(y_{12})}$ in $\asu_{w(y_{11})}^2$ is $q^{-r_1}$. Thus, $\tcuq$ is not a subalgebra.
\end{remark}

We also introduce notation for the structure coefficients of our new basis.

\begin{definition}
    Let $\ascu_{u,v}^w$ be the coefficient of $\asu_w$ in the product $\asu_u\asu_v$.
\end{definition}

\begin{remark}
Thus $\ascu_{u,v}^w=q^{\frac{l(u)+l(v)-l(w)-l(w_{\lambda})}{2}}\ascn_{u,v}^w$ for $u,v,w\in\hecia^{\lambda}$.
\end{remark}

\section{The Affine Cellular Structure}

We recall for reference the definition of an Affine Cellular algebra as given in \cite[Definition 2.1]{koenig2012affine}.

\begin{definition}\label{affine cellular}
    Let $R$ be a commutative Noetherian ring, let $A$ be an $R$-algebra, and let $i$ be an $R$-anti-involution on $A$. A 2-sided ideal $J$ in $A$ such that $i(J)=J$ is an affine cell ideal if there are
    \begin{itemize}
        \item a free $R$-module of finite rank $V$,
        \item a finitely generated commutative $R$-algebra $B$ with $R$-anti-involution $\sigma$,
        \item a left $A$-module structure on $\Delta=V\otimes_{R} B$ such that $\Delta$ is an $A$-$B$-bimodule with the regular right $B$-module structure,
    \end{itemize}
    such that, if we define a right $A$-module structure on $\Delta'=B\otimes_{R} V$ by $xa=\tau^{-1} (i(a)\tau(x))$ where $\tau:\Delta'\rightarrow\Delta$, $b\otimes v\mapsto v\otimes b$, there is an isomorphism of $A$-$A$-bimodules $\alpha: J\rightarrow \Delta\otimes_B\Delta'=V\otimes_{R} B\otimes_{R} V$ making the following diagram commute:
    \[\begin{tikzcd}
        J & {V\otimes_{R} B\otimes_{R} V} \\
        J & {V\otimes_{R} B\otimes_{R} V}
        \arrow["\alpha", from=1-1, to=1-2]
        \arrow["{v\otimes b\otimes v' \mapsto v'\otimes\sigma(b)\otimes v}", from=1-2, to=2-2]
        \arrow["i"', from=1-1, to=2-1]
        \arrow["\alpha"', from=2-1, to=2-2]
    \end{tikzcd}\]

    $A$ is affine cellular if there is an $R$-module decomposition $A=\bigoplus_{k=1}^KJ_k'$ such that, for all $k$, we have
    \begin{enumerate}
        \item $i(J_k')=J_k'$
        \item $J_k=\bigoplus_{k'=1}^kJ_{k'}'$ is a 2-sided ideal in $A$
        \item $J_k'=J_k/J_{k-1}$ is an affine cell ideal in $A/J_{k-1}$.
    \end{enumerate} 
\end{definition}

\begin{remark}
Recall from \cite[Proposition 2.2]{koenig2012affine} that an affine cell ideal $J$ inherits a product from $A$, which via $\alpha$ gives a product on $V\otimes_RB\otimes_RV$, and this product is of the form $(u\otimes b\otimes v)(u'\otimes b'\otimes v')=u\otimes b(v,u')b'\otimes v'$ for some $R$-bilinear map $V\times V\rightarrow B$.
\end{remark}

In this section we will show that $\hecpmq$ and $\schpmq$, together with their involution $\iota$, are affine cellular over $\zpmq$. The last remaining result is an integral version of the following well-known result.

\begin{prop}\label{normalised actions commute}
    Let $\klu_{QP}^u\in\schpmq$, let $\klu_{P\emptyset}^v,\klu_{Q\emptyset}^y\in\Psi^{\lambda}$, and let $\asu_{\emptyset\emptyset}^{w}\in\schipm^{\lambda}$. Then
    
    \begin{align*}
        \sum_{\klu_{Q\emptyset}^x\in\Psi^{\lambda}}\scn_{uPv}^{Qx\emptyset}\ascn_{x ,w}^{ y}=\sum_{\klu_{P\emptyset}^x\in\Psi^{\lambda}}\scn_{u Px}^{Qy\emptyset}\ascn_{v ,w}^{x}.
    \end{align*}
\end{prop}

\begin{proof}
By \cref{a is finite} (3) and \cref{hecke and shcur cell relations} we have $a(v)=a(y)$. 
    
Write $u=u'\omega^i$, $v=v'\omega^j$, $w=w'\omega^{k}$, and $y=y'\omega^l$ for $u',v',w',y'\in\affw$. Also write $v''=\omega^i v'\omega^{-i}$, $w''=\omega^{i+j} w'\omega^{-(i+j)}$, and $w'''=\omega^j w'\omega^{-j}$. Note in particular that by \cref{a is finite} (3) we have $a(v)=a(v')=a(v'')$ and $a(y')=a(y)$.
    
Hence, writing $x=x'\omega^m$ for $x\in\affw$ and noting that $a(x)=a(x')$ by \cref{a is finite} (3), we have by \cref{structure constant facts} and \cref{affine asymptotic structure} that

\begin{align*}
    \sum_{a(x)=a(v)}\scn_{u,v}^x\ascn_{x,w}^y &=\sum_{a(x')=a(v'')}\scn_{u',v''}^{x'}\ascn_{x',w''}^{y'}\delta_{i+j+k,l} \\
    \sum_{a(x)=a(v)}\scn_{u,x}^y\ascn_{v,w}^x&=\sum_{a(x')=a(v')}\scn_{u',\omega^{i}x'\omega^{-i}}^{y'}\ascn_{v',w'''}^{x'}\delta_{i+j+k,l}.
\end{align*}

In the latter equation, performing the substitution $x''=\omega^{i} x'\omega^{-i}$ and again noting by \cref{a is finite} (3) that $a(x')=a(x'')$, we have by \cref{affine asymptotic structure} that
\begin{align*}
    \sum_{a(x')=a(v')}\scn_{u',\omega^{i}x'\omega^{-i}}^{y'}\ascn_{v',w'''}^{x'}\delta_{i+j+k,l}&=\sum_{a(x'')=a(v'')}\scn_{u',x''}^{y'}\ascn_{v'',w''}^{x''}\delta_{i+j+k,l}
\end{align*}
Finally, by \cite[Equation 2.4(d)]{lusztig1987cells} we have
\begin{align*}
    \sum_{a(x')=a(v'')}\scn_{u',v''}^{x'}\ascn_{x',w''}^{y'}\delta_{i+j+k,l}
    &=\sum_{a(x'')=a(v'')}\scn_{u',x''}^{y'}\ascn_{v'',w''}^{x''}\delta_{i+j+k,l}.
\end{align*}
Putting this all together gives:
\begin{align*}
    \sum_{a(x)=a(v)}\scn_{u,v}^x\ascn_{x,w}^y &=\sum_{a(x')=a(v'')}\scn_{u',v''}^{x'}\ascn_{x',w''}^{y'}\delta_{i+j+k,l} \\
    &=\sum_{a(x'')=a(v'')}\scn_{u',x''}^{y'}\ascn_{v'',w''}^{x''}\delta_{i+j+k,l} \\
    &=\sum_{a(x')=a(v')}\scn_{u',\omega^{i}x'\omega^{-i}}^{y'}\ascn_{v',w'''}^{x'}\delta_{i+j+k,l} \\
    &=\sum_{a(x)=a(v)}\scn_{u,x}^y\ascn_{v,w}^x.
\end{align*}

Now, note that if $\scn_{u,v}^x\neq 0$ then $x\leq_{L}v$, and so by \cite[Corollary 1.9(b)]{lusztig1987cells} and \cref{cells on the affine subgroup} we have $x\sim_{L}v$ and hence $\klu_x\in\Gamma^{\lambda}$. Similarly, if $\scn_{u,x}^y\neq 0$ then $y\leq_{L}x$, and the same reasoning shows $\klu_x\in\Gamma^{\lambda}$. Thus we get
    \begin{align*}
        \sum_{\klu_x\in\Gamma^{\lambda}}\scn_{u,v}^x\ascn_{x,w}^y=\sum_{\klu_x\in\Gamma^{\lambda}}\scn_{u,x}^y\ascn_{v,w}^x.
    \end{align*}
    
But by \cref{schur algebra structure coefficients} we have that
\begin{align*}
    \scn_{uPv}^{Qx\emptyset}\ascn_{x w}^{y}=p_{P}^{-1}q^{\frac{l(w_P)}{2}} \scn_{u,v}^x\ascn_{x,w}^y
\end{align*}
 and
 \begin{align*}
 \scn_{u Px}^{Qy\emptyset}\ascn_{v w}^{x}=p_{P}^{-1}q^{\frac{l(w_P)}{2}}\scn_{u,x}^y\ascn_{v,w}^x.    
 \end{align*}
 
Hence, to finish the proof, it suffices to show that we can restrict the sums to only $x$ which are maximal length $\weyl_Q\backslash\weyl$ and $\weyl_P\backslash\weyl$ coset representatives respectively. But we know that $u$ and $v$ are maximal length representatives of these respective cosets, which by \cref{kl basis and idempotents} means exactly that the left action of $s\in Q$ and $s\in P$ on $\kln_u$ and $\kln_v$ respectively is multiplication by $q$. But acting on these by right multiplication does not change this property, so the sum is indeed only nonzero for such $x$.
\end{proof}

For each $\lambda\vdash n$, fix some indexing $\Gamma_1,\dots,\Gamma_{n_{\lambda}}$ of the left cells in $\hecpmq^{\lambda}$ such that $\Gamma_{1}=\Gamma^{\lambda}$. Then extend this to an indexing $\Psi_1,\dots,\Psi_{m_{\lambda}}$ of the left cells in $\schpmq^{\lambda}$ such that $\Gamma_{i}\subseteq\Phi_i$.

\begin{thm}\label{bimodule actions commute}
Let $\klu_{QP}^u\in\schpmq$, let $\klu_{P\emptyset}^v,\klu_{Q\emptyset}^y\in\Psi^{\lambda}$ such that $\klu_{P\emptyset}^v\in \iota\left(\Psi_i\right)$ and $\klu_{Q\emptyset}^y\in \iota\left(\Psi_j\right)$, and let $\asu_{w}\in B$. For $\klu_{R\emptyset}^x\in \iota\left(\Psi_k\right)$, let $\klu_{\emptyset \emptyset}^{x_k}$ be the unique element of $\iota\left(\Psi^{\lambda}\right)$ such that $\klu_{\emptyset\emptyset}^{x_k}=q^m\phi_{\iota\left(\Psi_k\right)}\left(\klu_{R\emptyset}^x\right)$ for some $m\in\frac{1}{2}\z$. Then
 \begin{align*}
    \sum_{\klu_{Q\emptyset}^x\in\Psi^{\lambda}\cap\iota\left(\Psi_j\right)}\scu_{uPv}^{Qx\emptyset}\ascu_{x_j,w}^{y_j}=\sum_{\klu_{P\emptyset}^x\in\Psi^{\lambda}\cap \iota\left(\Psi_i\right)}\scu_{u Px}^{Qy\emptyset}\ascu_{v_i,w}^{x_i}.
\end{align*}
\end{thm}

\begin{proof}
For $1\leq k\leq n_{\lambda}$, let $\klu_{\emptyset\emptyset}^{w^k}$ be the unique element of $\Psi_k\cap\iota\left(\Psi_k\right)$ such that $\klu_{\emptyset\emptyset}^{w}=q^m\phi_{\Psi_k\cap\iota\left(\Psi_k\right)}\left(\klu_{\emptyset\emptyset}^{w^k}\right)$ for some $m\in\frac{1}{2}\z$.
Taking \cref{normalised actions commute} and summing over all $w_k$ gives
    \begin{align*}
        \sum_{k=1}^{n_{\lambda}}\sum_{\klu_{Q\emptyset}^x\in\Psi^{\lambda}}\scn_{uPv}^{Qx\emptyset}\ascn_{x w_k}^{y}=\sum_{k=1}^{n_{\lambda}}\sum_{\klu_{P\emptyset}^x\in\Gamma^{\lambda}}\scn_{uPx}^{Qy\emptyset}\ascn_{v w_k}^{x}.
    \end{align*}

However, by \cref{normalised asymptotic cell}, we have $\gamma_{x w_k}^{y}=0$ unless $k=j$ and $\gamma_{v w_k}^{x}=0$ unless $k=i$. Hence this reduces to
\begin{align*}
    \sum_{\klu_{Q\emptyset}^x\in\Psi^{\lambda}}\scn_{uPv}^{Qx\emptyset}\ascn_{x w_j}^{y}=\sum_{\klu_{P\emptyset}^x\in\Psi^{\lambda}}\scn_{uPx}^{Qy\emptyset}\ascn_{v w_i}^{x}.
\end{align*}

But, by \cref{normalised asymptotic cell} again, if $\klu_{Q\emptyset}^x\notin\iota\left(\Psi_j\right)$ then $\gamma_{x w_j}^{y}=0$. Similarly, if $\klu_{P\emptyset}^x\notin\iota\left(\Psi_i\right)$ then $\gamma_{v w_i}^{x}=0$. Hence we may further reduce the sum to 
\begin{align*}
    \sum_{\klu_{Q\emptyset}^x\in\Psi^{\lambda}\cap\iota\left(\Psi_j\right)}\scn_{uPv}^{Qx\emptyset}\ascn_{x w^j}^{y}=\sum_{\klu_{P\emptyset}^x\in\Psi^{\lambda}\cap\iota\left(\Psi_i\right)}\scn_{uPx}^{Qy\emptyset}\ascn_{v w^i}^{x}.
\end{align*}

Now, by \cref{normalised asymptotic cell} one more time we have that $\ascn_{x_j w}^{ y_j}=\ascn_{x w^j}^{y}$ and $\ascn_{v w^i}^{x}=\ascn_{v_i w}^{ x_i}$, and hence
\begin{align*}
\sum_{\klu_{Q\emptyset}^x\in\Psi^{\lambda}\cap\iota\left(\Psi_j\right)}\scn_{uPv}^{Qx\emptyset}\ascn_{x_j w}^{ y_j} 
=\sum_{\klu_{P\emptyset}^x\in\Psi^{\lambda}\cap\iota\left(\Psi_i\right)}\scn_{uPx}^{Qy\emptyset}\ascn_{v_i w}^{ x_i}.
\end{align*}

Now, in the left sum, as both $\klu_{Q\emptyset}^x,\klu_{Q\emptyset}^y\in \iota\left(\Psi_j\right)$, by construction of $\phi_{\iota\left(\Psi_j\right)}$ we have that $l(x)-l(x_j)=l(y)-l(y_j)$. Analogously, in the right sum, $l(x)-l(x_i)=l(v)-l(v_i)$. Using this, we may now conclude:
\begin{align*}
    &\sum_{\klu_{Q\emptyset}^x\in\Psi^{\lambda}\cap\iota\left(\Psi_j\right)}\scu_{uPv}^{Qx\emptyset}\ascu_{x_j,w}^{y_j}\\ =&\sum_{\klu_{Q\emptyset}^x\in\Psi^{\lambda}\cap\iota\left(\Psi_j\right)}q^{\frac{l(u)+l(v)+l(w)+l(x_j)-l(x)-l(y_j)-l(w_{\lambda})}{2}}\scn_{uPv}^{Qx\emptyset}\ascn_{x_j,w}^{y_j} \\
    =&q^{\frac{l(u)+l(v)+l(w)-l(y)-l(w_{\lambda})}{2}}\sum_{\klu_{Q\emptyset}^x\in\Psi^{\lambda}\cap\iota\left(\Psi_j\right)}\scn_{uPv}^{Q\emptyset}\ascn_{x_j,w}^{y_j} \\
    =&q^{\frac{l(u)+l(v)+l(w)-l(y)-l(w_{\lambda})}{2}}\sum_{\klu_{P\emptyset}^x\in\Psi^{\lambda}\cap\iota\left(\Psi_i\right)}\scn_{uPx}^{Qy\emptyset}\ascn_{v_i w}^{ x_i} \\
    =&\sum_{\klu_{P\emptyset}^x\in\Psi^{\lambda}\cap\iota\left(\Psi_i\right)}q^{\frac{l(u)+l(v_i)+l(w)+l(x)-l(x_i)-l(y)-l(w_{\lambda})}{2}}\scn_{uPx}^{Qy\emptyset}\ascn_{v_i,w}^{x_i} \\
    =&\sum_{\klu_{P\emptyset}^x\in\Psi^{\lambda}\cap\iota\left(\Psi_i\right)}\scu_{uPx}^{Qy\emptyset}\ascu_{v_i,w}^{x_i}.
\end{align*}

\end{proof}

We are now ready to put all the pieces together.

\begin{thm}\label{schur affine cellular}
$\hecpmq$ and $\schpmq$ are affine cellular over $\zpmq$. More specifically:
\begin{itemize}
    \item The cell ideals are respectively the $\hecpmq^{\lambda}$ and the $\schpmq^{\lambda}$
    \item The order on the cell ideals is any total order extending the decreasing dominance order
    \item The affine cell ideal structure on $\hecpmq^{\lambda}$ and $\schpmq^{\lambda}$ is given by
    \begin{itemize}
        \item the free $\zpmq$-modules $V=\left<v_{1},\dots, v_{n_{\lambda}}\right>$ and $U=\left<u_{1},\dots, u_{m_{\lambda}}\right>$ of rank
        \begin{align*}
           n_{\lambda}=\frac{n!}{\prod_{i}(\mu_i!)} \quad \text{and} \quad m_{\lambda}=\prod_{i=1}^{n-1}\left(\frac{n!}{(n-i)!i!}\right)^{\lambda_i-\lambda_{i+1}}
        \end{align*}
        respectively, where $\mu$ is the dual partition of $\lambda$
        \item the finitely generated commutative algebra 
        \begin{align*}
            B=\tcupmq=\zpmq\left[\asu_{w(y_{ij})},\asu_{w(y_{im_i})}^{-1}\middle|1\leq i\leq p,1\leq j\leq m_i\right]
        \end{align*}
        where $m_i$ is the $i$th nonzero value in the sequence $\lambda_r-\lambda_{r+1}$
        \item the isomorphisms
        \begin{align*}
            \alpha^{-1}:V\otimes_{\zpmq}B\otimes_{\zpmq}V&\rightarrow \hecpmq^{\lambda} \\ \beta^{-1}:U\otimes_{\zpmq}B\otimes_{\zpmq}U&\rightarrow \schpmq^{\lambda}
        \end{align*}
        given by
        \begin{align*}
            v_i\otimes \asu_{w}\otimes v_j&\mapsto \phi_{\Gamma_{j}\cap\iota\left(\Gamma_i\right)}\left(\klu_{w}\right) \\
            u_i\otimes \asu_{w}\otimes u_j&\mapsto \phi_{\Psi_{j}\cap\iota\left(\Psi_i\right)}\left(\klu_{\emptyset\emptyset}^{w}\right).
        \end{align*}
    \end{itemize}
\end{itemize}
\end{thm}

\begin{proof}
Fix some indexing $\lambda^1,\dots,\lambda^K$ of $\Lambda$ such that if $\lambda^i>\lambda^j$ then $i<j$.

We have $\hecpmq=\bigoplus_{k=1}^K\hecpmq^{\lambda^k}$ and $\schpmq=\bigoplus_{k=1}^K\schpmq^{\lambda^k}$. Furthermore, by \cref{cells respect involution}, we have $\iota\left(\hecpmq^{\lambda^{k}}\right)=\hecpmq^{\lambda^{k}}$ and $\iota\left(\schpmq^{\lambda^{k}}\right)=\schpmq^{\lambda^{k}}$, and by \cref{cells are ideals} the submodules $\hecpmq^k=\bigoplus_{k'=1}^k\hecpmq^{\lambda^{k'}}$ and $\schpmq^k=\bigoplus_{k'=1}^k\schpmq^{\lambda^{k'}}$ are two-sided ideals in $\hecpmq$ and $\schpmq$ respectively.

To show that the $\hecpmq^{\lambda^{k}}$ and $\schpmq^{\lambda^{k}}$ are affine cell ideals in $\hecpmq/\hecpmq^{k-1}$ and $\schpmq/\schpmq^{k-1}$ respectively, it suffices to show they are affine cell ideals in the larger quotients $\hecpmq/\hecpmq^{>\lambda_{k}}$ and $\schpmq/\schpmq^{>\lambda_{k}}$ respectively. Henceforth we just write $\lambda$ for $\lambda_k$.

By \cref{number of left cells}, $U$ and $V$ have rank equal to the number of left cells in $\hecpmq^{\lambda}$ and $\schpmq^{\lambda}$ respectively. Write the standard basis for these spaces as $v_1,\dots,v_{n_{\lambda}}$ and $u_1,\dots\,u_{m_{\lambda}}$ respectively. Consider $V$ as a submodule of $U$ via $v_i\mapsto u_i$. Also write $B=\tcupmq$, noting that then $B=\left[\asu_{w(y_{ij})},\asu_{w(y_{im_i})}^{-1}\middle|1\leq i\leq p,1\leq j\leq m_i\right]$ by \cref{properties of B}.

By \cref{cells are module iso}, $\alpha^{-1}$ and $\beta^{-1}$ are module isomorphisms, justifying the notation $\alpha^{-1}$ and $\beta^{-1}$. Observe that $\alpha$ is the restriction to $\hecpmq^{\lambda}$ of $\beta$. Furthermore, again by \cref{cells are module iso}, $\iota\beta^{-1}(u_i\otimes \asu_w\otimes u_j)=\beta^{-1}(u_j\otimes \sigma(\asu_w)\otimes u_j)$.

By construction, $\alpha^{-1}\left(V\otimes_{\zpmq}B\otimes v_j\right)=\Gamma_j$, which by \cref{left cells are left ideals} is a left $\hecpmq/\hecpmq^{>\lambda}$-ideal, and similarly $\beta^{-1}\left(U\otimes_{\zpmq}B\otimes u_j\right)=\Phi_j$, which is a left $\schpmq/\schpmq^{>\lambda}$-ideal. Using \cref{cells are module iso} another time, we have $\phi_{\Gamma_j}(\alpha^{-1}(v_i\otimes\asu_{w}\otimes v_j))=\alpha^{-1}(v_i\otimes\asu_{w}\otimes v_1)$ and $\phi_{\Psi_j}\left(\beta^{-1}(u_i\otimes\asu_{w}\otimes u_j)\right)=\beta^{-1}(u_i\otimes\asu_{w}\otimes u_1)$, and by \cref{left cells are iso} these are left $\hecpmq/\hecpmq^{>\lambda}$-module isomorphisms and left $\schpmq/\schpmq^{>\lambda}$-module isomorphisms respectively. Thus the respective left actions of $\hecpmq/\hecpmq^{>\lambda}$ and $\schpmq/\schpmq^{>\lambda}$ on $V\otimes_{\zpmq} B\otimes_{\zpmq} V$ and $U\otimes_{\zpmq} B\otimes_{\zpmq}U$ come from well-defined left actions on $V\otimes_{\zpmq} B$ and $U\otimes_{\zpmq} B$ respectively.

Finally, by \cref{bimodule actions commute}, the right regular $B$-action and left $\schpmq/\schpmq^{>\lambda}$-action on $\Psi^{\lambda}=\beta^{-1}(U\otimes_{\zpmq}B\otimes u_1)$ commute. Hence, by restriction, so do the right regular $B$-action and left $\hecpmq/\hecpmq^{>\lambda}$-action on $\Gamma^{\lambda}=\beta^{-1}(V\otimes_{\zpmq}B\otimes v_1)$.
\end{proof}

We finish by observing that this affine cellular structure has further idempotence properties. Write $\hec_{p_P^{-1}}=\hecpmq\otimes_{\zpmq}\zpmq\left[p_P^{-1}\right]$.

\begin{thm}\label{schur cells idempotent}
The two-sided ideal $\schpmq^{\lambda}$ in $\schpmq/\schpmq^{>\lambda}$ is generated by the nonzero idempotent $\klu_{P_{\lambda}P_{\lambda}}^{w_{\lambda}}$.

Furthermore, two-sided ideal $\hec_{p_P^{-1}}^{\lambda}$ in $\hec_{p_P^{-1}}/\hec_{p_P^{-1}}^{>\lambda}$ is generated by the nonzero idempotent $\frac{1}{p_{P_{\lambda}}}\klu_{w_{\lambda}}$.
\end{thm}

\begin{proof}
 By \cref{w lambda in h lambda} we have $\klu_{w_{\lambda}}\in\hecany^{\lambda}$, and so by \cref{hecke and shcur cell relations} we also have $\klu_{P_{\lambda}P_{\lambda}}^{w_{\lambda}}\in\schany^{\lambda}$.

 Now, let $P\subseteq\fingen$. Then $\klu_{PP}^{w_P}$ is a sum of $T_{PP}^{y}$ for $y\leq w_P$ of maximal length in $\weyl_{P}$. But the only such $y$ is $w_P$, and $P_{w,w}=1$ for any $w$, so $\klu_{PP}^{w_P}=T_{PP}^{w_P}=1_{x_{P}\hecany}$, the identity map on $x_{P}\hecany$, and so is idempotent.
 
 Now, $\klu_{PP}^{w_P}$ being idempotent exactly means that $\scu_{w_PPw_P}^{PvP}=\delta_{vw}$. Hence $\scu_{w_Pw_P}^v=p_{P}\scu_{w_PPw_P}^{PvP}=p_P\delta_{vw_P}$, so we thus have $\klu_{w_P}\klu_{w_P}=p_P\klu_{w_P}$, and so $\frac{1}{p_P}\klu_{w_P}$ is also idempotent in $\hec_{p_P^{-1}}$. Furthermore,
 \begin{align*}
     \scu_{w_P\emptyset w_P}^{Pv\emptyset}=\scu_{w_P\emptyset w_P}^{\emptyset vP}=\scu_{w_Pw_P}^{v}=p_P\delta_{vw_P}
 \end{align*}
 and so $\klu_{P\emptyset}^{w_P}\klu_{w_P}\klu_{\emptyset P}^{w_P}=p_P^2\klu_{PP}^{w_P}$. Hence if $\klu_{P_{\lambda}P_{\lambda}}^{w_{\lambda}}$ generates $\schpmq^{\lambda}$ then $\klu_{w_{P_{\lambda}}}$ generates $\hec_{p_P^{-1}}^{\lambda}$.
 
 Now consider a general basis element $\klu_{QP}^w\in\schpmq^{\lambda}$. Write
 \begin{align*}
 \beta\left(\klu_{QP}^w\right)=q^{m_1}u_i\otimes b\otimes u_j\quad \text{and} \quad \beta\left(\klu_{P_{\lambda}P_{\lambda}}^{w_{\lambda}}\right)=u_k\otimes 1\otimes u_k.     
 \end{align*}
 Define $q^{m_2}\klu_{QP_{\lambda}}^u=\beta^{-1}\left(u_i\otimes b\otimes u_k\right)$ and $q^{m_3}\klu_{P_{\lambda}P}^v=\beta^{-1}\left(u_k\otimes 1\otimes u_j\right)$. Observe that in fact $q^{m_1+m_2+m_3}=1$. Now, since $\klu_{P_{\lambda}P_{\lambda}}^{w_{\lambda}}$ is idempotent, we have $(u_k\otimes 1\otimes u_k)^2=u_k\otimes 1\otimes u_k$, and so $(u_k,u_k)=1$. Hence in $\schpmq/\schpmq^{>\lambda}$ we have
 \begin{align*}
     \klu_{QP_{\lambda}}^u\klu_{P_{\lambda}P_{\lambda}}^{w_{\lambda}}\klu_{P_{\lambda}P}^v
     &= q^{-(m_2+m_3)}\beta^{-1}\left((u_i\otimes b\otimes u_k)(u_k\otimes 1\otimes u_k)(u_k\otimes 1\otimes u_j)\right) \\
     &= q^{m_1}\beta^{-1}\left(u_i\otimes b\otimes u_j\right) \\
     &= \klu_{QP}^w.
 \end{align*}
\end{proof}

\begin{thm}
Let $R$ be a Noetherian domain and a $\zpmq$-algebra. Then $\sch_R=\schpmq\otimes_{\zpmq}R$ satisfies the conditions of \cite[Theorem 4.4]{koenig2012affine}. Hence, the derived category of $\sch_R$ admits a stratification by the derived categories of the
\begin{align*}
    \tcupmq\otimes_{\zpmq}R=R\left[\asu_{w(y_{ij})},\asu_{w(y_{im_i})}^{-1}\middle|1\leq i\leq p,1\leq j\leq m_i\right],
\end{align*}
and if $R$ has finite global dimension then so does $\sch_R$.

If furthermore $p_{\fingen}$ is invertible in $R$, then the analogous claims hold for $\hec_R=\hecpmq\otimes_{\zpmq}R$.
\end{thm}

\begin{proof}
 By \cite[Lemma 2.4]{koenig2012affine} and \cref{schur affine cellular}, $\sch_R$ is affine cellular, with affine cell ideals $\sch_R^{\lambda}=\schpmq^{\lambda}\otimes_{\zpmq}R\cong U_{R}\otimes\left(\tcupmq\otimes_{\zpmq}R\right)\otimes_R U_R$, where $U_R=U\otimes_{\zpmq}R$. But $\tcupmq\otimes_{\zpmq}R=R\left[\asu_{w(y_{ij})},\asu_{w(y_{im_i})}^{-1}\middle|1\leq i\leq p,1\leq j\leq m_i\right]$ is a domain and hence has zero Jacobson radical. Furthermore, it has global dimension $R+\sum_{i=1}^pm_i=R+\lambda_1$ by standard results: see for example \cite[Theorem 7.5.3 (iii,iv)]{mcconnell2001noncommutative}.
 
 The arguments for $\hec_R$ are the same, noting that if $P\subseteq\fingen$ then $p_P$ divides $p_{\fingen}$.
\end{proof}

\printbibliography

\end{document}